\documentclass[a4paper,12pt]{article}%
\usepackage{amsmath}
\usepackage{amsfonts}
\usepackage{amssymb}
\usepackage{graphicx}
\usepackage{algorithm}
\usepackage{algorithmic}
\usepackage{booktabs}
\usepackage{url}
\usepackage{hyperref}
\usepackage{color}%
\providecommand{\U}[1]{\protect\rule{.1in}{.1in}}

\newcommand\email[1]{{\href{mailto:#1}{\nolinkurl{#1}}}}
\newtheorem{theorem}{Theorem}

\newtheorem{corollary}[theorem]{Corollary}

\newtheorem{lemma}[theorem]{Lemma}

\newtheorem{remark}[theorem]{Remark}

\newenvironment{proof}[1][Proof]{\noindent\textbf{#1.} }{\ \rule{0.5em}{0.5em}}
\begin{document}

\title{Computing the first eigenpair of the $p$-Laplacian via inverse iteration of
sublinear supersolutions}
\author{{\small \textbf{Rodney Josu\'{e} BIEZUNER$^{\text{\,a}}$, Jed
BROWN$^{\text{\,b}}$, Grey ERCOLE$^{\text{\,a}}$, Eder Marinho
MARTINS$^{\text{\,c}}$}}\thanks{\textit{E-mail addresses:} rodney@mat.ufmg.br
(R. J. Biezuner), jed@59A2.org (J. Brown), grey@mat.ufmg.br (G. Ercole),
eder@iceb.ufop.br (E. Martins).}\\{\small {$^{\mathrm{a}}$} \textit{Departamento de Matem\'{a}tica - ICEx,
Universidade Federal de Minas Gerais,}}\\{\small \textit{Av. Ant\^{o}nio Carlos 6627, Caixa Postal 702, 30161-970, Belo
Horizonte, MG, Brazil.}} \\{\small {$^{\mathrm{b}}$} \textit{Laboratory of Hydraulics, Hydrology, and
Glaciology (VAW), ETH Z\"urich, 8092 Z\"urich, Switzerland,}}\\{\small \textit{Mathematics and Computer Science Division, Argonne National
Laboratory, Argonne, IL 60439, USA.}} \\{\small {$^{\mathrm{c}}$} \textit{Departamento de Matem\'{a}tica - ICEB,
Universidade Federal de Ouro Preto,}}\\{\small \textit{Campus Universit\'ario Morro do Cruzeiro, 35400-000, Ouro
Preto, MG, Brazil. }}}
\maketitle

\begin{abstract}
We introduce an iterative method for computing the first eigenpair $\left(
\lambda_{p},e_{p}\right)  $ for the $p$-Laplacian operator with homogeneous
Dirichlet data as the limit of $\left(  \mu_{q,}u_{q}\right)  $ as
$q\rightarrow p^{-}$, where $u_{q}$ is the positive solution of the sublinear
Lane-Emden equation $-\Delta_{p}u_{q}=\mu_{q}u_{q}^{q-1}$ with the same
boundary data. The method is shown to work for any smooth, bounded domain.
Solutions to the Lane-Emden problem are obtained through inverse iteration of
a super-solution which is derived from the solution to the torsional creep
problem. Convergence of $u_{q}$ to $e_{p}$ is in the $C^{1}$-norm and the rate
of convergence of $\mu_{q}$ to $\lambda_{p}$ is at least $O\left(  p-q\right)
$. Numerical evidence is presented.

\end{abstract}


\section{Introduction}

In this paper we develop an iterative method to obtain the first eigenpair
$\left(  \lambda_{p},e_{p}\right)  $ of the eigenvalue problem
\begin{equation}
\left\{
\begin{array}
[c]{ll}%
-\Delta_{p}u=\lambda\left\vert u\right\vert ^{p-2}u & \text{ \ \ in }\Omega,\\
u=0 & \text{\ \ \ on }\partial\Omega,
\end{array}
\right.  \label{01}%
\end{equation}
where $\Delta_{p}u:=\operatorname{div}\left(  \left\vert \nabla u\right\vert
^{p-2}\nabla u\right)  $, $p>1$, is the $p$-Laplacian operator and
$\Omega\subset\mathbb{R}^{N}$, $N\geqslant2$, is any smooth, bounded domain.
The $p$-Laplacian equation appears in several mathematical models in fluid
dynamics, such as in the modelling of non-Newtonian fluids and glaciology
\cite{ADO, BR, GR, PR}, turbulent flows \cite{DT}, climatology \cite{DH}
nonlinear diffusion (where it is called the $N$-diffusion equation; see
\cite{Phillip} for the original article and \cite{GZ} for some current
developments), flow through porous media \cite{SW}, power law materials
\cite{AC} and in the study of torsional creep \cite{Kawohl}.

The first eigenvalue $\lambda_{p}$ of (\ref{01}) is variationally
characterized by%
\[
\lambda_{p}=\min_{u\in W_{0}^{1,p}\left(  \Omega\right)  /\left\{  0\right\}
}R\left(  u\right)  >0
\]
where $R$ is the Rayleigh quotient%
\[
R\left(  u\right)  =\dfrac{\int_{\Omega}\left\vert \nabla u\right\vert ^{p}%
dx}{\int_{\Omega}\left\vert u\right\vert ^{p}dx}.
\]
The first eigenfunction $e_{p}$ of (\ref{01}) is characterized by the fact
that the minimum of $R$ is attained at $e_{p}$, so that%
\[
\lambda_{p}=\dfrac{\int_{\Omega}\left\vert \nabla e_{p}\right\vert ^{p}%
dx}{\int_{\Omega}e_{p}^{p}dx}.
\]
It is well-known that $\lambda_{p}$ is isolated and simple, and that the
corresponding eigenfunction $e_{p}\in C^{1,\alpha}\left(  \overline{\Omega
}\right)  $ can be taken positive. Since $R$ is homogeneous, we may assume
$\Vert e_{p}\Vert_{\infty}=1$, where $\Vert\cdot\Vert_{\infty}$ stands for the
$L^{\infty}$-norm.

In the one-dimensional case the first eigenpair $\left(  \lambda_{p}%
,e_{p}\right)  $ is explicitly determined by solving the corresponding ODE
boundary value problem. If $\Omega=(a,b)$, then $\lambda_{p}=\left(  \pi
_{p}/\left(  b-a\right)  \right)  ^{p-1}$ and $e_{p}=(p-1)^{-1/p}\sin
_{p}\left(  \pi_{p}\left(  x-a\right)  /\left(  b-a\right)  \right)  $, where
$\pi_{p}:=2(p-1)^{1/p}\int_{0}^{1}(1-s^{p})^{-1/p}ds$ and $\sin_{p}$\ is a
$2\pi_{p}$-periodic function that generalizes the classical sine function (see
\cite{BEM2, Lindqvist}).

When $p=2$, we have $\Delta_{p}=\Delta$, the Laplacian operator, whose first
eigenpair $\left(  \lambda_{p},e_{p}\right)  $ is well-known for domains with
simple geometry (that is, domains which admit some kind of symmetry); for more
general domains it can be determined by several numerical methods (see
\cite{BEM3, Descloux-Tolley, Guidotti-Lambers, Heuveline, Kuttler-Sigillito}).
However, if $p\neq2$ and $N\geqslant2$, the first eigenpair is not explicitly
known even for simple symmetric domains such as a square or a ball, and there
are few available numerical methods to deal directly with the eigenproblem
(\ref{01}) in these domains (see \cite{BEM1, bognar2003, Bognar, BR,
Lefton,Yau}).

On the other hand, several numerical methods are available to solve
homogeneous Dirichlet problems for the (Poisson) $p$-Laplacian equation in the
form
\[
-\Delta_{p}u=f\left(  x\right)
\]
when $f$ depends only on $x\in\Omega$ (see \cite{AinsworthKay, Barrett,
Bermejo, Diening, Droniou, Veeser}). This fact motivated the development of an
inverse iteration method by some of the authors for finding the first
eigenpair in \cite{BEM1}. If $\Omega$ is a $N$-dimensional ball, the
convergence of the method was established and numerical evidence for its
applicability when $\Omega$ is a 2-dimensional square were also presented. In
the special case of the Laplacian operator, the method was proved to work in
general domains and can be also used to obtain other eigenpairs (see
\cite{BEM3}). However, since the method was based on the iteration of the
nonlinear $p$-Laplacian equation in (\ref{01}), the difficulties in dealing
with the nonlinearity on the right-hand side of the equation prevented showing
that the method works in any domain and any $p>1$.

In this work we consider a different inverse iteration approach, also based on
the solution of the Poisson $p$-Laplacian equation, but built around an
eigenproblem which has a sublinear nonlinearity on its right-hand side. This
type on nonlinearity is more manageable and we are able to prove that the
iterative method works for any smooth, bounded domain. It is based on
obtaining positive solutions $v_{\mu,q}$ for the Lane-Emden type problem
\begin{equation}
\left\{
\begin{array}
[c]{ll}%
-\Delta_{p}v=\mu\left\vert v\right\vert ^{q-2}v & \text{ \ \ in }\Omega,\\
v=0 & \text{\ \ \ on }\partial\Omega.
\end{array}
\right.  \label{03}%
\end{equation}
After rescaling, $\mu$ and $v_{\mu,q}$ produce a family of pairs $\left\{
\left(  \mu_{q},u_{q}\right)  \right\}  _{1<q<p}$ converging to the first
eigenpair $\left(  \lambda_{p},e_{p}\right)  $ when $q\rightarrow p^{-}$, the
convergence $u_{q}\rightarrow e_{p}$ being in $C^{1}\left(  \overline{\Omega
}\right)  $. We will now describe the method in more detail.

It is well known that for each fixed $\mu>0$, problem (\ref{03}) has a unique
solution $v_{\mu,q}$, if $1<q<p$ (see \cite{Huang}). If $q=p$, we have the
$p$-Laplacian eigenvalue problem. If $q>p$, positive solutions of (\ref{03})
usually are not unique. A nonuniqueness result for ring-shaped domains is
given in \cite{Azorero} when $q$ is close to the Sobolev critical exponent
$p^{\ast}$ ($p^{\ast}=Np/\left(  N-p\right)  $, if $1<p<N$, and $p^{\ast
}=\infty$, if $p\geqslant N$). On the other hand, as proved in \cite{AY},
positive solutions are unique when $\Omega$ is a ball, while for general
bounded domains the uniqueness of positive solutions that reach the minimum
energy (ground states) was established in \cite{Drabek} under the conditions
$1<p<N$ and $1<q<p^{\ast}$.

Now, in order to construct the approximating sequence to the first eigenpair,
first choose any $\mu>0$ and a sequence $\left(  q_{n}\right)  $, $1<q_{n}<p$,
such that $q_{n}\rightarrow p^{-}$. It is important to notice that $\mu$ need
not to be taken close to $\lambda_{p}$. This point is crucial, since good
\textit{a priori} estimates for $\lambda_{p}$ are hard to obtain. For each
$q_{n}$ we need to solve the Lane-Emden problem (\ref{03}) in order to find
$v_{\mu,q_{n}}$, which is a degenerate nonlinear problem almost as hard to
solve as the eigenvalue problem for the $p$-Laplacian (\ref{01}) itself. In
order to obtain the solutions $v_{\mu,q_{n}}$ we first solve the much easier
\textit{torsional creep problem}%
\begin{equation}
\left\{
\begin{array}
[c]{ll}%
-\Delta_{p}\phi=1 & \text{ \ \ in }\Omega,\\
\phi=0 & \text{\ \ \ on }\partial\Omega.
\end{array}
\right.  \label{04}%
\end{equation}
Then compute $k_{p}=\left\Vert \phi\right\Vert _{\infty}^{1-p}$ and set
\begin{equation}
\phi_{0}=\left(  \dfrac{\mu}{k_{p}}\right)  ^{\frac{1}{p-q_{n}}}\dfrac{\phi
}{\Vert\phi\Vert_{\infty}}. \label{phizero}%
\end{equation}
$\phi_{0}$ is a supersolution to (\ref{03}). One immediately sees that the
easiest choice is $\mu=k_{p}$, so that $\phi_{0}=\phi/\Vert\phi\Vert_{\infty}%
$. Now apply an inverse iteration to $\phi_{0}$, finding a sequence of
iterates $\left(  \phi_{m}\right)  $ which satisfy%
\begin{equation}
\left\{
\begin{array}
[c]{ll}%
-\Delta_{p}\phi_{m+1}=\mu\phi_{m}^{q_{n}-1} & \text{ \ \ in }\Omega,\\
\phi_{m+1}=0 & \text{\ \ \ on }\partial\Omega.
\end{array}
\right.  \label{phim}%
\end{equation}
This can be done by a number of numerical methods. Finite volume based methods
are presented in \cite{Andreianov, Droniou}; finite element based methods are
also available (see \cite{HLL} and the references therein). After a
pre-established tolerance limit has been reached at some $\phi_{m}$, where $m$
is a function of $\mu$ and $q_{n}$, set%
\[
v_{\mu,q_{n}}=\phi_{m}%
\]
and define $u_{q_{n}}$ and $\mu_{q_{n}}$ as
\[
\mu_{q_{n}}:=\frac{\mu}{\left\Vert v_{\mu,q_{n}}\right\Vert _{\infty}%
^{p-q_{n}}}\text{ \ and \ }u_{q_{n}}:=\frac{v_{\mu,q_{n}}}{\left\Vert
v_{\mu,q_{n}}\right\Vert _{\infty}}.
\]
In Theorem \ref{convergence} we show that $\mu_{q_{n}}\rightarrow\lambda_{p}$
and $u_{q_{n}}\rightarrow e_{p}$ in $C^{1}\left(  \overline{\Omega}\right)  $
when $q_{n}\rightarrow p^{-}$. Choosing a value for $q$ close to $p$ will give
an approximation for the first eigenpair of the $p$-Laplacian. The procedure
is summarized in Algorithm 1 below.

\begin{algorithm}[!ht]
\caption{Inverse iteration for the first $p$-Laplacian eigenpair $(\lambda_p,e_p)$} \small{
\begin{algorithmic}[1]
\STATE \textbf{set} $\mu$ \qquad\qquad\qquad\qquad\qquad\qquad\qquad\qquad\qquad\qquad\qquad\ \!\! (an arbitrary positive number)
\STATE \textbf{set} $q$ \qquad\qquad\qquad\qquad\qquad\qquad\qquad\qquad\qquad\qquad\qquad\ ($q$ should be chosen close to $p$)
\STATE \textbf{solve} $-\Delta_{p}\phi=1 \text{ in }\Omega,\text{ \ }\phi=0 \text{ on }\partial\Omega$ \qquad\qquad\qquad\qquad\  (torsion function)
\STATE \textbf{set} $\phi_0=\left(  \mu /k_{p}\right)  ^{\frac{1}{p-q}}\phi/\left\Vert \phi\right\Vert_{\infty}$
\qquad\qquad\qquad\qquad\qquad\qquad\ \ (supersolution)
\FOR{$m=0,1,2,\ldots$}
\STATE \textbf{solve} $-\Delta_{p}\phi_{m+1}=\mu\phi_{m}^{q-1} \text{ in }\Omega,\text{ \ }
\phi_{m+1}=0 \text{ on }\partial\Omega$ \qquad(Inverse iteration sequence)
\ENDFOR
\RETURN
$\mu/\left\Vert \phi_{m+1}\right\Vert_{\infty}^{p-q}$ \qquad\qquad\qquad\qquad\qquad\qquad\qquad\!\ \ \ (first eigenvalue $\lambda_p$)
\RETURN
$\phi_{m+1}/\left\Vert \phi_{m+1}\right\Vert_{\infty}$ \qquad\qquad\qquad\qquad\qquad\qquad\qquad (first eigenfunction $e_p$)
\end{algorithmic}
}
\end{algorithm}

However, we are able to produce a much more robust algorithm which is also
easier to apply in practice. In Algorithm 2 below one does not need to use an
arbitrary parameter $\mu$ nor to compute the value of the constant $k_{p}$.
Normalization of the iterates at each step increases robustness and thus it
should be the algorithm of choice.

\begin{algorithm}[!ht]
\caption{Inverse iteration for the first $p$-Laplacian eigenpair $(\lambda_p,e_p)$ with normalization} \small{
\begin{algorithmic}[1]
\STATE \textbf{set} $q$ \qquad\qquad\qquad\qquad\qquad\qquad\qquad\qquad\qquad\qquad\qquad\qquad\qquad\ \ \ ($q$ should be chosen close to $p$)
\STATE \textbf{solve} $-\Delta_{p}\phi_0=1 \text{ in }\Omega,\text{ \ }\phi_0=0 \text{ on }\partial\Omega$ \qquad\qquad\qquad\qquad\qquad\qquad\qquad\!\!\!\!\!\! (torsion function)
\FOR{$m=0,1,2,\ldots$}
\STATE \textbf{solve} $-\Delta_{p}\phi_{m+1}=\left( \phi_m/\left\Vert \phi_m\right\Vert_{\infty} \right)^{q-1} \text{ in }\Omega,\text{ \ }
\phi_{m+1}=0 \text{ on }\partial\Omega$ \qquad(Inverse iteration sequence)
\ENDFOR
\RETURN
$1/\left\Vert \phi_{m+1}\right\Vert_{\infty}^{p-1}$ \qquad\qquad\qquad\qquad\qquad\qquad\qquad\qquad\qquad\qquad\ \!\!\!\! (first eigenvalue $\lambda_p$)
\RETURN
$\phi_{m+1}/\left\Vert \phi_{m+1}\right\Vert_{\infty}$ \qquad\qquad\qquad\qquad\qquad\qquad\qquad\qquad\qquad\ \ (first eigenfunction $e_p$)
\end{algorithmic}
}
\end{algorithm}

The outline of the paper is as follows. In Section 2 we present some
preliminary results that will be used in the sequel. The sequences of
approximates for both algorithms are built in Section 3 and the proof of their
convergence to the first eigenpair is given in Section 4. In Section
\ref{Some numerical results} we present some numerical results for the unit
ball of dimensions $N=2,3,4$ using the first algorithm, and for the
two-dimensional square and the three-dimensional cube and torus using the
second algorithm.

The main advantage of the method presented here, besides its applicability to
general domains, is that approximations to both $\lambda_{p}$ and $e_{p}$ are
obtained with the desired precision by an iterative process which is
numerically simple and, in the case of a ball, also explicit.

\section{Preliminary results\label{preliminaries}}

In this section we state simple versions of some results on the $p$-Laplacian.
We begin with the following comparison principle (see \cite{Damascelli} for a
more general version).

\begin{lemma}
For $i\in\left\{  1,2\right\}  $, let $h_{i}\in C\left(  \overline{\Omega
}\right)  $ and $u_{i}\in W^{1,p}\left(  \Omega\right)  $ be such that
$-\Delta_{p}u_{i}=h_{i}$ in $\Omega.$ If $h_{1}\leq h_{2}$ in $\Omega$ and
$u_{1}\leqslant u_{2}$ on $\partial\Omega,$ then $u_{1}\leqslant u_{2}$ in
$\Omega.$
\end{lemma}

The following result is a simple version of a general result proved in the
classical paper \cite{Lieberman} of Lieberman.

\begin{theorem}
\textrm{\cite[Thm 1]{Lieberman}} \label{Liebb} Suppose that $u\in
W^{1,p}\left(  \Omega\right)  $ is a weak solution of the Dirichlet problem%
\[
\left\{
\begin{array}
[c]{ll}%
-\Delta_{p}u=f(x,u) & \text{ \ \ in }\Omega,\\
u=0 & \text{\ \ \ on }\partial\Omega
\end{array}
\right.
\]
where $f$ is a continuous function such that%
\[
\left\vert f(x,\xi)\right\vert \leqslant\Lambda\text{ \ for all }(x,\xi
)\in\Omega\times\lbrack-M,M]
\]
for positive constants $\Lambda$ and $M.$

If $\left\Vert u\right\Vert _{\infty}\leqslant M$, then there exists
$0<\alpha<1$, depending only on $\Lambda,$ $p$ and the dimension $N$, such
that $u\in C^{1,\alpha}\left(  \overline{\Omega}\right)  $; moreover we have
\[
\left\Vert u\right\Vert _{C^{1,\alpha}\left(  \overline{\Omega}\right)
}\leqslant C,
\]
where $C$ is a positive constant that depends only on $\Lambda,$ $p,$ $N$ and
$M.$
\end{theorem}

Thus, denoting by $\phi$ is the solution of the torsional creep problem
(\ref{04}) in the domain $\Omega$, one can easily verify using (\ref{Rtorsion}%
) and the comparison principle in balls that $0<\phi\leqslant M$ \ in $\Omega$
for some positive constant $M$. Hence, Theorem \ref{Liebb} implies that
$\phi\in C^{1,\alpha}\left(  \overline{\Omega}\right)  $ for some $0<\alpha<1$.

For the next lemma set
\begin{equation}
k_{p}:=\left\Vert \phi\right\Vert _{\infty}^{1-p}>0. \label{05}%
\end{equation}

\begin{lemma}
\label{lbound}$k_{p}\leqslant\lambda_{p}.$
\end{lemma}

\begin{proof}
Let $e_{p}$ be the first eigenfuncion associated with $\lambda_{p}$ satisfying
$\left\Vert e_{p}\right\Vert _{\infty}=1$ in $\Omega.$ Since
\[
\left\{
\begin{array}
[c]{ll}%
-\Delta_{p}e_{p}=\lambda_{p}e_{p}^{p-1}\leqslant\lambda_{p}=-\Delta_{p}\left(
\lambda_{p}^{\frac{1}{p-1}}\phi\right)  & \text{ \ \ in }\Omega,\\
e_{p}=0=\lambda_{p}^{\frac{1}{p-1}}\phi & \text{\ \ \ on }\partial\Omega,
\end{array}
\right.
\]
it follows from the comparison principle that
\[
0<e_{p}\leqslant\lambda_{p}^{\frac{1}{p-1}}\phi\text{\ \ \ in }\Omega.
\]
Hence,%
\[
1=\left\Vert e_{p}\right\Vert _{\infty}\leqslant\lambda_{p}^{\frac{1}{p-1}%
}\left\Vert \phi\right\Vert _{\infty},
\]
from what follows our claim.\vspace{0.5cm}
\end{proof}

\begin{remark}
It follows from Picone's identity (see \cite{AH}) that, in fact, the
inequality is strict, that is, $k_{p}<\lambda_{p}$ (for details, see
\cite[Lemma 8.1]{BEZ}).
\end{remark}

The following result is well-known and follows from Theorem \ref{Liebb}.

\begin{theorem}
\label{compactness}Let $-\Delta_{p}^{-1}:C^{1}\left(  \overline{\Omega
}\right)  \rightarrow W_{0}^{1,p}\left(  \Omega\right)  $ be the operator
defined as follows: for each $v\in C^{1}\left(  \overline{\Omega}\right)  $
let $-\Delta_{p}^{-1}v:=u\in W_{0}^{1,p}\left(  \Omega\right)  $ be the unique
solution of the Dirichlet problem
\[
\left\{
\begin{array}
[c]{ll}%
-\Delta_{p}u=v & \text{ \ \ in }\Omega,\\
u=0 & \text{\ \ \ on }\partial\Omega.
\end{array}
\right.
\]
Then $-\Delta_{p}^{-1}$ is continuous and compact. Moreover, $-\Delta_{p}%
^{-1}v\in C^{1,\alpha}\left(  \overline{\Omega}\right)  $ for each $v\in
C^{1}\left(  \overline{\Omega}\right)  .$
\end{theorem}

In the remainder of the paper $\left(  \lambda_{p},e_{p}\right)  $ denotes the
first eigenpair of (\ref{01}), $\phi$ denotes the torsion function of $\Omega$
and $k_{p}:=\left\Vert \phi\right\Vert _{\infty}^{1-p}.$

\section{Construction of the sequence of
approximates\label{The iterative method}}

As mentioned before, if $1<q<p$, then for each $\mu>0$ the Lane-Emden problem%
\begin{equation}
\left\{
\begin{array}
[c]{ll}%
-\Delta_{p}v=\mu\left\vert v\right\vert ^{q-2}v & \text{ \ \ in }\Omega,\\
v=0 & \text{\ \ \ on }\partial\Omega,
\end{array}
\right.  \label{Laneq}%
\end{equation}
has a unique positive solution $v_{\mu,q}$, which can be obtained via standard
variational, and therefore \textit{non-constructive}, arguments. The existence
and uniqueness of solutions of (\ref{Laneq}) in the case $1<q<p$ implies that
the map $\mu\mapsto v_{\mu,q}$ is well-defined and monotone, in the sense that
$\mu_{1}<\mu_{2}$ implies $v_{\mu_{1},q}<v_{\mu_{2},q}$ in $\Omega$, since
$v_{\mu_{1},q}=\left(  \mu_{1}/\mu_{2}\right)  ^{1/\left(  p-q\right)  }%
v_{\mu_{2},q}$ for any $\mu_{1},$ $\mu_{2}>0$.

The basis of our constructive method is given by

\begin{theorem}
\label{muq}Suppose $1<q<p.$ For each $\mu>0$ the unique positive solution
$v_{\mu,q}\in C^{1,\alpha}\left(  \overline{\Omega}\right)  \cap W_{0}%
^{1,p}\left(  \Omega\right)  $ of $(\ref{Laneq})$ satisfies
\begin{equation}
0<\left(  \frac{\mu}{\lambda_{p}}\right)  ^{\frac{1}{p-q}}e_{p}\leqslant
v_{\mu,q}\leqslant\left(  \frac{\mu}{k_{p}}\right)  ^{\frac{1}{p-q}}\frac
{\phi}{\left\Vert \phi\right\Vert _{\infty}}\text{ \ \ in }\Omega. \label{06}%
\end{equation}

Moreover, $v_{\mu,q}\ $ is the limit, in the $C^{1}\left(  \overline{\Omega
}\right)  $ norm, of the sequence $\left\{  v_{n}\right\}  \subset
C^{1,\alpha}\left(  \overline{\Omega}\right)  \cap W_{0}^{1,p}\left(
\Omega\right)  $ iteratively defined by%

\begin{equation}
v_{1}:=\left(  \frac{\mu}{k_{p}}\right)  ^{\frac{1}{p-q}}\frac{\phi
}{\left\Vert \phi\right\Vert _{\infty}} \label{v0}%
\end{equation}
and, for $n\geqslant1$,
\begin{equation}
\left\{
\begin{array}
[c]{ll}%
-\Delta_{p}v_{n+1}=\mu v_{n}^{q-1} & \text{ \ \ in }\Omega,\\
v_{n+1}=0 & \text{\ \ \ on }\partial\Omega.
\end{array}
\right.  \label{iterativeq}%
\end{equation}

\end{theorem}

\begin{proof}
Define $\underline{v}_{\mu,q}:=me_{p}$ and $\overline{v}_{\mu,q}:=\dfrac
{M\phi}{\left\Vert \phi\right\Vert _{\infty}}$ where
\[
m:=\left(  \dfrac{\mu}{\lambda_{p}}\right)  ^{\frac{1}{p-q}}\text{ \ \ and
\ \ }M:=\left(  \dfrac{\mu}{k_{p}}\right)  ^{\frac{1}{p-q}}.
\]
We have
\begin{equation}
-\Delta_{p}\underline{v}_{\mu,q}\leqslant\mu\underline{v}_{\mu,q}^{q-1}\text{
\ \ and \ \ }-\Delta_{p}\overline{v}_{\mu,q}\geqslant\mu\overline{v}_{\mu
,q}^{q-1}\text{ \ \ in }\Omega. \label{subsup}%
\end{equation}
Indeed, in $\Omega$ in we have
\[
-\Delta_{p}\underline{v}_{\mu,q}=\lambda_{p}\underline{v}_{\mu,q}%
^{p-1}=\lambda_{p}\underline{v}_{\mu,q}^{p-q}\underline{v}_{\mu,q}%
^{q-1}=\lambda_{p}\left(  me_{p}\right)  ^{p-q}\underline{v}_{\mu,q}%
^{q-1}\leqslant\lambda_{p}m^{p-q}\underline{v}_{\mu,q}^{q-1}=\mu\underline
{v}_{\mu,q}^{q-1}%
\]
and%
\[
-\Delta_{p}\overline{v}_{\mu,q}=k_{p}M^{p-1}=k_{p}M^{p-q}M^{q-1}\geqslant
k_{p}M^{p-q}\left(  \frac{M\phi}{\left\Vert \phi\right\Vert _{\infty}}\right)
^{q-1}=\mu\overline{v}_{\mu,q}^{q-1}.
\]
Since $\underline{v}_{\mu,q}=0=\overline{v}_{\mu,q}$ on $\partial\Omega$ the
inequalities in (\ref{subsup}) mean that $\underline{v}_{\mu,q}$ and
$\overline{v}_{\mu,q}$ are, respectively, sub- and supersolutions for
(\ref{Laneq}).

Moreover, $\underline{v}_{\mu,q}$ and $\overline{v}_{\mu,q}$ are ordered, that
is $\underline{v}_{\mu,q}\leqslant\overline{v}_{\mu,q}$ in $\Omega$. For,
since $k_{p}\leqslant\lambda_{p}$, we have%
\begin{align*}
\lambda_{p}m^{p-1}  &  =\lambda_{p}\left(  \frac{\mu}{\lambda_{p}}\right)
^{\frac{p-1}{p-q}}\\
&  =\mu^{\frac{p-1}{p-q}}\left(  \frac{1}{\lambda_{p}}\right)  ^{\frac
{q-1}{p-q}}\leqslant\mu^{\frac{p-1}{p-q}}\left(  \frac{1}{k_{p}}\right)
^{\frac{q-1}{p-q}}=k_{p}\left(  \frac{\mu}{k_{p}}\right)  ^{\frac{p-1}{p-q}%
}=k_{p}M^{p-1},
\end{align*}
whence
\[
-\Delta_{p}\underline{v}_{\mu,q}=\lambda_{p}\underline{v}_{\mu,q}%
^{p-1}\leqslant\lambda_{p}m^{p-1}\leqslant k_{p}M^{p-1}=-\Delta_{p}%
\overline{v}_{\mu,q}%
\]
in $\Omega$. Thus, since $\underline{v}_{\mu,q}=\overline{v}_{\mu,q}=0$ on
$\partial\Omega,$ we obtain $\underline{v}_{\mu,q}\leqslant\overline{v}%
_{\mu,q}$ in $\Omega$ by applying the comparison principle.

Since $u\mapsto\mu u^{q-1}$ is increasing and $\underline{v}_{\mu,q}%
\leqslant\overline{v}_{\mu,q}$ in $\Omega$, the comparison principle also
implies that the sequence $\left\{  v_{n}\right\}  $ defined by the iterative
process (\ref{iterativeq}) starting with the supersolution $\overline{v}%
_{\mu,q}$ satisfies
\[
\underline{v}_{\mu,q}\leqslant v_{n+1}\leqslant v_{n}\leqslant\overline
{v}_{\mu,q}\text{ \ in }\Omega.
\]

Hence, $v_{n}$ converges to a function $v_{\mu,q}$ a.e. in $\Omega$. Since
$\left\Vert v_{n}\right\Vert _{\infty}\leqslant\left\Vert \overline{v}_{\mu
,q}\right\Vert _{\infty}=M$, it follows from Theorem \ref{Liebb} that
$\left\{  v_{n}\right\}  \subset C^{1,\alpha}\left(  \overline{\Omega}\right)
$ for some $0<\alpha<1$ (which does not depend on $n$) and that
\[
\left\Vert v_{n}\right\Vert _{C^{1,\alpha}\left(  \overline{\Omega}\right)
}\leqslant C
\]
for some positive constant $C$ which is independent of $n.$

Thus, from Arzela-Ascoli theorem we conclude that $v_{n}\rightarrow v$ in the
$C^{1}$ norm.

Now, the continuity of the operator $-\Delta_{p}^{-1}:$ $C^{1}\left(
\overline{\Omega}\right)  \rightarrow W_{0}^{1,p}\left(  \Omega\right)  $
permits passing to the limit in (\ref{iterativeq}), which yields that
$v_{\mu,q}\in C^{1}\left(  \overline{\Omega}\right)  \cap W_{0}^{1,p}\left(
\Omega\right)  $ is a solution of (\ref{Laneq}) satisfying
\[
0<\underline{v}_{\mu,q}\leqslant v_{\mu,q}\leqslant\overline{v}_{\mu,q}\text{
\ in }\Omega,
\]
proving (\ref{06}). The regularity $v_{\mu,q}\in C^{1,\alpha}\left(
\overline{\Omega}\right)  $ follows from Theorem \ref{Liebb}.
\end{proof}

This iterative process is also known as inverse iteration since $v_{n+1}%
=-\Delta_{p}^{-1} (\mu v^{q-1}_{n}).$ It is essentially the sub- and
supersolution method starting with the supersolution $\overline{v}_{\mu,q}$;
the solution $v_{\mu,q}$ that it produces is characterized as the maximal
solution between $\underline{v}_{\mu,q}\ $and $\overline{v}_{\mu,q}.$

If one starts the iteration with the subsolution then one obtains an
increasing sequence converging to the minimal solution between $\underline
{v}_{\mu,q}\ $and $\overline{v}_{\mu,q}.$ Because of the uniqueness this
minimal solution coincides with $v_{\mu,q}.$ However, in order to compute the
minimal solution from this iterative process, it is necessary to know
\textit{a priori} a subsolution, which is exactly one of the unknowns that we
wish to find by applying the method.

On the other hand the supersolution $\overline{v}_{\mu,q}$ is easily
obtainable since it involves the solution of the simpler problem (\ref{04}).

For example, if $\Omega=B_{R}(x_{0})$, the ball centered at $x_{0}%
\in\mathbb{R}^{N}$ with radius $R>0$, then it is easy to verify (see also
below) that the \textit{torsion function} $\phi$ is the radial function
\begin{equation}
\phi\left(  r\right)  =\frac{p-1}{pN^{\frac{1}{p-1}}}\left(  R^{\frac{p}{p-1}%
}-\left\vert r\right\vert ^{\frac{p}{p-1}}\right)  ,\text{ \ }r=\left\vert
x-x_{0}\right\vert \leqslant R.\label{Rtorsion}%
\end{equation}
We then obtain that%
\begin{equation}
k_{p}=\left\Vert \phi\right\Vert _{\infty}^{1-p}=\frac{N}{R^{p}}\left(
\frac{p}{p-1}\right)  ^{p-1}\label{kpball}%
\end{equation}
and%
\[
\overline{v}_{\mu,q}\left(  r\right)  =\left(  \frac{\mu}{k_{p}}\right)
^{\frac{1}{p-q}}\left(  1-\left\vert r\right\vert ^{\frac{p}{p-1}}\right)
=\mu^{\frac{1}{p-q}}\left(  \frac{p-1}{pN^{\frac{1}{p-1}}}\right)
^{\frac{p-1}{p-q}}\left(  1-\left\vert r\right\vert ^{\frac{p}{p-1}}\right)
\]
where $r=\left\vert x-x_{0}\right\vert .$

In this case the sequence $v_{n}$ converging to $v_{\mu,q}$ is given
recursively by the formula%
\begin{equation}
v_{n+1}\left(  r\right)  =\int_{r}^{R}\left(  \int_{0}^{\theta}\left(
\frac{s}{\theta}\right)  ^{N-1}\mu v_{n}(s)^{q-1}ds\right)  ^{\frac{1}{p-1}%
}d\theta\label{vnradial}%
\end{equation}
where $v_{0}\left(  r\right)  =\overline{v}_{\mu,q}\left(  r\right)  $.

This integral formula follows from the more general fact: the Poisson problem%
\[
\left\{
\begin{array}
[c]{l}%
-\Delta_{p}u=f(|x|)\text{ \ in \ }B_{R}(x_{0})\\
u=0\text{ \ on \ }\partial B_{R}(x_{0})
\end{array}
\right.
\]
is equivalent to the ODE boundary value problem%
\[
\left\{
\begin{array}
[c]{l}%
-\left(  r^{N-1}\left\vert u^{\prime}\right\vert ^{p-2}u^{\prime}\right)
^{\prime}=r^{N-1}f(r),\text{ }0<r<R\\
u^{\prime}(0)=0=u(R)
\end{array}
\right.
\]
for radial solutions $u=u(r),$ $r=\left\vert x-x_{0}\right\vert .$ Hence,
after two integrations of the ODE taking into account the boundary conditions
one obtains the following integral expression%
\begin{equation}
u(r)=\int_{r}^{R}\left(  \int_{0}^{\theta}(\frac{s}{\theta})^{N-1}%
f(s)ds\right)  ^{\frac{1}{p-1}}d\theta\label{greenball}%
\end{equation}
for the solution $u=u(\left\vert x-x_{0}\right\vert )$ of the Poisson problem.
In particular, when $f(r)\equiv1\ $this integral form might be simplified in
order to find the expression (\ref{Rtorsion}) for the torsion function of
$B_{R}.$

In our method, in order to compute the first eigenpair $\left(  \lambda
_{p},e_{p}\right)  $, we fix a positive value $\mu>0$ and choose $q$ close to
$p^{-}.$ Then, we apply the inverse iteration of Theorem \ref{muq} starting
with the supersolution%
\[
\overline{v}_{\mu,q}=\left(  \frac{\mu}{k_{p}}\right)  ^{\frac{1}{p-q}}%
\frac{\phi}{\left\Vert \phi\right\Vert _{\infty}}%
\]
to obtain approximations for the function $v_{\mu,q}.$ Hence,
\[
\dfrac{\mu}{\left\Vert v_{\mu,q}\right\Vert _{\infty}^{p-q}}\rightarrow
\lambda_{p}\text{ \ \ and \ }\dfrac{v_{\mu,q}}{\left\Vert v_{\mu,q}\right\Vert
_{\infty}}\rightarrow e_{p}\text{ (in the }C^{1}\text{ norm)}%
\]
as $q\rightarrow p,$ a result that we prove in the next section.

For the construction of the normalized sequence of Algorithm 2 one needs the
following result:

\begin{theorem}
\label{C2}Suppose $1<q<p.$ Then the normalized sequence $\left\{
w_{n}/\left\Vert w_{n}\right\Vert _{\infty}\right\}  $ where $w_{n}$ is
defined by%
\[
w_{0}:=1\ \ \text{and}\ \left\{
\begin{array}
[c]{ll}%
\underset{}{-\Delta_{p}w_{n+1}=\left(  \dfrac{w_{n}}{\left\Vert w_{n}%
\right\Vert _{\infty}}\right)  ^{q-1}} & \text{ \ \ in }\Omega\\
w_{n+1}=0 & \text{ \ \ on }\partial\Omega
\end{array}
\right.
\]
converges in the $C^{1}(\overline{\Omega})$ norm to $v_{q}/\left\Vert
v_{q}\right\Vert _{\infty}\ $where $v_{q}\in C^{1,\alpha}(\overline{\Omega
})\cap W_{0}^{1,p}(\Omega)$ is the solution of (\ref{Laneq}) with $\mu=k_{p}.$
\end{theorem}

\begin{proof}
Let $\left\{  v_{n}\right\}  $ be the sequence defined by
\begin{equation}
v_{1}:=\frac{\phi}{\left\Vert \phi\right\Vert _{\infty}}\text{ \ and
\ }\left\{
\begin{array}
[c]{ll}%
-\Delta_{p}v_{n+1}=k_{p}v_{n}^{q-1} & \text{ \ \ in }\Omega\\
v_{n+1}=0 & \text{ \ \ on }\partial\Omega.
\end{array}
\right.  \label{vmu==1}%
\end{equation}
It follows from Theorem \ref{muq} that the sequence $\left\{  v_{n}\right\}  $
is decreasing and converges in the $C^{1}(\overline{\Omega})$ norm to the
solution $v_{q}\in C^{1,\alpha}(\overline{\Omega})\cap W_{0}^{1,p}(\Omega)$ of
the Lane-Emden problem%
\begin{equation}
\left\{
\begin{array}
[c]{ll}%
-\Delta_{p}v=k_{p}v^{q-1} & \text{ \ \ in }\Omega\\
v=0 & \text{ \ \ on }\partial\Omega
\end{array}
\right.  \label{mu==kp}%
\end{equation}

Since $w_{1}=\phi$ we have that
\[
w_{2}=k_{p}^{-\frac{1}{p-1}}v_{2}%
\]
and, in particular $\dfrac{w_{2}}{\left\Vert w_{2}\right\Vert _{\infty}%
}=\dfrac{v_{2}}{\left\Vert v_{2}\right\Vert _{\infty}}.$ In fact, this follows
from the comparison principle:%
\begin{align*}
-\Delta_{p}(k_{p}^{-\frac{1}{p-1}}v_{2})  &  =k_{p}^{-1}\left(  -\Delta
_{p}v_{2}\right) \\
&  =k_{p}^{-1}k_{p}v_{1}^{q-1}=\left(  \frac{\phi}{\left\Vert \phi\right\Vert
}\right)  ^{q-1}=\left(  \frac{w_{1}}{\left\Vert w_{1}\right\Vert _{\infty}%
}\right)  ^{q-1}=-\Delta_{p}w_{2.}%
\end{align*}
Repeating this procedure we obtain
\begin{align*}
-\Delta_{p}\left(  k_{p}^{-\frac{1}{p-1}}\left\Vert v_{2}\right\Vert _{\infty
}^{-\frac{q-1}{p-1}}v_{3}\right)   &  =k_{p}^{-1}\left\Vert v_{2}\right\Vert
_{\infty}^{-(q-1)}\left(  -\Delta_{p}v_{3}\right) \\
&  =k_{p}^{-1}\left\Vert v_{2}\right\Vert _{\infty}^{-(q-1)}k_{p}v_{2}^{q-1}\\
&  =\left(  \frac{v_{2}}{\left\Vert v_{2}\right\Vert _{\infty}}\right)
^{q-1}=\left(  \frac{w_{2}}{\left\Vert w_{2}\right\Vert _{\infty}}\right)
^{q-1}=-\Delta_{p}w_{3},
\end{align*}
that is%
\[
w_{3}=k_{p}^{-\frac{1}{p-1}}\left\Vert v_{2}\right\Vert _{\infty}^{-\frac
{q-1}{p-1}}v_{3}%
\]
and%
\[
\frac{w_{3}}{\left\Vert w_{3}\right\Vert _{\infty}}=\frac{v_{3}}{\left\Vert
v_{3}\right\Vert _{\infty}}.
\]
Therefore, by an induction argument we conclude that%
\[
w_{n+1}=k_{p}^{-\frac{1}{p-1}}\left\Vert v_{n}\right\Vert _{\infty}%
^{-\frac{q-1}{p-1}}v_{n+1}\text{ \ and \ }\frac{w_{n+1}}{\left\Vert
w_{n+1}\right\Vert _{\infty}}=\frac{v_{n+1}}{\left\Vert v_{n+1}\right\Vert
_{\infty}}%
\]
for all $n\geqslant2$. Hence, it follows from Theorem \ref{muq} that%
\[
\frac{w_{n+1}}{\left\Vert w_{n+1}\right\Vert _{\infty}}=\frac{v_{n+1}%
}{\left\Vert v_{n+1}\right\Vert _{\infty}}\rightarrow\frac{v_{q}}{\left\Vert
v_{q}\right\Vert _{\infty}}.
\]

\end{proof}

\section{Convergence of the method\label{The convergence}}

\begin{theorem}
\label{convergence}\textit{For }$\mu>0$\textit{ and for each }$1<q<p$\textit{
set }%
\begin{equation}
u_{q}:=\dfrac{v_{\mu,q}}{\left\Vert v_{\mu,q}\right\Vert _{\infty}},
\label{uq}%
\end{equation}
\textit{ where }$v_{\mu,q}\in C^{1,\alpha}\left(  \overline{\Omega}\right)
$\textit{ is the unique positive solution of }(\ref{Laneq})\textit{, and }%
\begin{equation}
\mu_{q}:=\dfrac{\mu}{\left\Vert v_{\mu,q}\right\Vert _{\infty}^{p-q}}.
\label{lamdaq}%
\end{equation}
\textit{Then }$\mu_{q}\rightarrow\lambda_{p}$\textit{ and }$u_{q}\rightarrow
e_{p}$\textit{ in }$C^{1}\left(  \overline{\Omega}\right)  $\textit{ as
}$q\rightarrow p^{-}$\textit{.}
\end{theorem}

\begin{proof}
Since $\left\Vert u_{q}\right\Vert _{\infty}=1$ and%
\[
-\Delta_{p}u_{q}=\frac{\mu}{\left\Vert v_{\mu,q}\right\Vert _{\infty}^{p-1}%
}v_{\mu,q}^{q-1}=\frac{\mu}{\left\Vert v_{\mu,q}\right\Vert _{\infty}^{p-q}%
}u_{q}^{q-1}=\mu_{q}u_{q}^{q-1},
\]
we have that $u_{q}$ is the unique solution of the problem%
\begin{equation}
\left\{
\begin{array}
[c]{ll}%
-\Delta_{p}u_{q}=\mu_{q}u_{q}^{q-1} & \text{\ \ \ in }\Omega,\\
u_{q}=0 & \text{\ \ \ on }\partial\Omega.
\end{array}
\right.  \label{11}%
\end{equation}
As a consequence of (\ref{06}) we have%
\[
\frac{\mu}{\lambda_{p}}\leqslant\left\Vert v_{\mu,q}\right\Vert _{\infty
}^{p-q}\leqslant\frac{\mu}{k_{p}},
\]%
\begin{equation}
0<\left(  \frac{k_{p}}{\lambda_{p}}\right)  ^{\frac{1}{p-q}}e_{p}\leqslant
u_{q}\leqslant\left(  \frac{\lambda_{p}}{k_{p}}\right)  ^{\frac{1}{p-q}}%
\frac{\phi}{\left\Vert \phi\right\Vert _{\infty}}\text{ \ \ in }%
\Omega\label{upositive}%
\end{equation}
and%
\begin{equation}
k_{p}\leqslant\mu_{q}\leqslant\lambda_{p}. \label{mulambda}%
\end{equation}
Since
\[
0\leqslant\mu_{q}u_{q}^{q-1}\leqslant\lambda_{p},
\]
it follows from Theorem \ref{Liebb} the existence of constants $0<\alpha<1$
and $C>0$ independent of $q$ such that $u_{q}\in C^{1,\alpha}\left(
\overline{\Omega}\right)  $ and%
\[
\left\Vert u_{q}\right\Vert _{C^{1,\alpha}\left(  \overline{\Omega}\right)
}\leqslant C\text{ \ \ for all }1<q<p.
\]
Using the compactness of the immersion $C^{1,\alpha}\left(  \overline{\Omega
}\right)  \hookrightarrow C^{1}\left(  \overline{\Omega}\right)  $, letting
$q_{n}\rightarrow p$ we get, up to a subsequence, $\mu_{q_{n}}\rightarrow
\lambda\in\lbrack k_{p},\lambda_{p}]$ and $u_{q_{n}}\rightarrow u$ in
$C^{1}\left(  \overline{\Omega}\right)  $. Taking the limit in (\ref{11}), we
conclude from Theorem \ref{compactness} that $u$ must satisfy%
\[
\left\{
\begin{array}
[c]{ll}%
-\Delta_{p}u=\lambda u^{p-1} & \text{\ \ \ in }\Omega,\\
u=0 & \text{\ \ \ on }\partial\Omega,
\end{array}
\right.
\]
and $\left\Vert u\right\Vert _{\infty}=1$, whence $\lambda=\lambda_{p}$ and
$u=e_{p}$ because $\lambda$ is an eigenvalue and $u\neq0$ is a corresponding
eigenfuntion that does not change the signal in $\Omega$ (note from
(\ref{upositive}) that $u>0$ in $\Omega$). Since these limits are always the
same, that is, do not depend on particular subsequences, this ends the proof.
\end{proof}

\begin{corollary}
Using the notation of Theorem \ref{C2}, it follows that%
\[
\frac{w_{n}}{\left\Vert w_{n}\right\Vert _{\infty}}\rightarrow e_{p}%
\]
in the $C^{1}(\overline{\Omega})$ norm and%
\[
\left\Vert w_{n}\right\Vert _{\infty}^{1-p}\rightarrow\lambda_{p}%
\]
as $n\rightarrow\infty$ and $q\rightarrow p^{-}.$
\end{corollary}

\begin{proof}
For each $1<q<p$ let $v_{q}$ denote the solution of the Lane-Emden problem
(\ref{mu==kp}). It follows from Theorem \ref{convergence} that $v_{q}%
/\left\Vert v_{q}\right\Vert _{\infty}$ converges in the $C^{1}$-norm to the
first eigenfunction $e_{p}$ and that $k_{p}\left\Vert v_{q}\right\Vert
_{\infty}^{q-p}\rightarrow\lambda_{p}.$

Using the notation of Theorem \ref{C2} we have%
\[
\lim_{q\rightarrow p^{-}}\lim_{n\rightarrow\infty}\frac{w_{n}}{\left\Vert
w_{n}\right\Vert _{\infty}}=\lim_{q\rightarrow p^{-}}\lim_{n\rightarrow\infty
}\frac{v_{n}}{\left\Vert v_{n}\right\Vert _{\infty}}=\lim_{q\rightarrow p^{-}%
}\frac{v_{q}}{\left\Vert v_{q}\right\Vert _{\infty}}=e_{p}.
\]
Moreover,
\[
\lim_{n\rightarrow\infty}\left\Vert w_{n+1}\right\Vert _{\infty}^{p-1}%
=\lim_{n\rightarrow\infty}k_{p}^{-1}\left\Vert v_{n+1}\right\Vert _{\infty
}^{p-1}\left\Vert v_{n}\right\Vert ^{1-q}=k_{p}^{-1}\left\Vert v_{q}%
\right\Vert _{\infty}^{p-q},
\]
and hence%
\[
\lim_{q\rightarrow p^{-}}\lim_{n\rightarrow\infty}\left\Vert w_{n+1}%
\right\Vert _{\infty}^{1-p}=\lambda_{p}.
\]

\end{proof}

Next we prove an error estimate in the approximation of $\lambda_{p}$ by
$\mu_{q}$ or, alternatively, by the scaled quotient
\[
\Lambda_{q}:=\mu\dfrac{\left\Vert v_{\mu,q}\right\Vert _{q}^{q}}{\left\Vert
v_{\mu,q}\right\Vert _{p}^{p}},
\]
where $\left\Vert \cdot\right\Vert _{r}$ denotes the norm of the $L^{r}\left(
\Omega\right)  ,$ that is, $\left\Vert w\right\Vert _{r}=\left(  \int_{\Omega
}\left\vert w\right\vert ^{r}dx\right)  ^{\frac{1}{r}}.$

The upper bound $\Lambda_{q}$ together with the lower bound $\mu_{q}$ allows
one to better control the accuracy of the approximation to $\lambda_{p}$.

\begin{theorem}
\textit{There holds:}

\begin{enumerate}
\item[(i)] $\lambda_{p}\leqslant\Lambda_{q}.$

\item[(ii)] $\Lambda_{q}\rightarrow\lambda_{p}$ as $q\rightarrow p^{-}.$

\item[(iii)] There exists a positive constant $K$ which does not depend on $q$
such that%
\begin{equation}
0\leqslant\max\left\{  \left(  \lambda_{p}-\mu_{q}\right)  ,\left(
\Lambda_{q}-\lambda_{p}\right)  \right\}  \leqslant K\left(  p-q\right)
\label{13}%
\end{equation}
for all $q$ sufficiently close to $p$, $q<p$.
\end{enumerate}
\end{theorem}

\begin{proof}
(i) follows directly from the variational characterization of $\lambda_{p}$
and (\ref{03}), since%

\[
\lambda_{p}\leqslant\frac{\left\Vert \nabla v_{\mu,q}\right\Vert _{p}^{p}%
}{\left\Vert v_{\mu,q}\right\Vert _{p}^{p}}=\dfrac{\mu\left\Vert v_{\mu
,q}\right\Vert _{q}^{q}}{\left\Vert v_{\mu,q}\right\Vert _{p}^{p}}=\Lambda
_{q}.
\]
In order to prove (ii) we note from Theorem \ref{convergence} that
\begin{equation}
\lim\limits_{q\rightarrow p^{-}}\left\Vert u_{q}\right\Vert _{q}^{q}%
=\lim\limits_{q\rightarrow p^{-}}\left\Vert u_{q}\right\Vert _{p}%
^{p}=\left\Vert e_{p}\right\Vert _{p}^{p}, \label{12}%
\end{equation}
since $u_{q}$ converges uniformly to $e_{p}$ when $q\rightarrow p^{-}$. Thus,
since%
\begin{equation}
\Lambda_{q}=\mu\dfrac{\left\Vert v_{\mu,q}\right\Vert _{q}^{q}}{\left\Vert
v_{\mu,q}\right\Vert _{p}^{p}}=\frac{\mu}{\left\Vert v_{\mu,q}\right\Vert
_{\infty}^{p-q}}\frac{\left\Vert u_{q}\right\Vert _{q}^{q}}{\left\Vert
u_{q}\right\Vert _{p}^{p}}=\mu_{q}\frac{\left\Vert u_{q}\right\Vert _{q}^{q}%
}{\left\Vert u_{q}\right\Vert _{p}^{p}}, \label{14}%
\end{equation}
we obtain
\[
\lim\limits_{q\rightarrow p^{-}}\Lambda_{q}=\left(  \lim\limits_{q\rightarrow
p^{-}}\mu_{q}\right)  \left(  \lim\limits_{q\rightarrow p^{-}}\frac{\left\Vert
u_{q}\right\Vert _{q}^{q}}{\left\Vert u_{q}\right\Vert _{p}^{p}}\right)
=\lambda_{p}.
\]

Now we prove error estimate (\ref{13}). It follows from (i) and
(\ref{mulambda}) that%
\[
\mu_{q}\leqslant\lambda_{p}\leqslant\Lambda_{q}.
\]
Hence,%
\[
0\leqslant\max\left\{  \left(  \lambda_{p}-\mu_{q}\right)  ,\left(
\Lambda_{q}-\lambda_{p}\right)  \right\}  \leqslant\Lambda_{q}-\mu_{q}.
\]
Thus, in order to prove (iii) we need only to bound $\Lambda_{q}-\mu_{q}.$ It
follows from (\ref{14}) that
\[
\Lambda_{q}-\mu_{q}=\mu_{q}\left(  \frac{\left\Vert u_{q}\right\Vert _{q}^{q}%
}{\left\Vert u_{q}\right\Vert _{p}^{p}}-1\right)  =\mu_{q}\frac{\int_{\Omega
}\left(  u_{q}^{q}-u_{q}^{p}\right)  dx}{\int_{\Omega}u_{q}^{p}dx}.
\]
Therefore,%
\begin{align*}
\Lambda_{q}-\mu_{q}  &  \leqslant\lambda_{p}\frac{\int_{\Omega}\left(
u_{q}^{q}-u_{q}^{p}\right)  dx}{\int_{\Omega}u_{q}^{p}dx}\\
&  \leqslant\frac{\lambda_{p}}{\int_{\Omega}u_{q}^{p}dx}\int_{\Omega}\left[
\max_{0\leqslant t\leqslant1}\left(  t^{q}-t^{p}\right)  \right]  dx\\
&  =\frac{\lambda_{p}\left\vert \Omega\right\vert }{\int_{\Omega}u_{q}^{p}%
dx}\left(  \frac{q}{p}\right)  ^{\frac{q}{p-q}}\frac{p-q}{p}\\
&  \leqslant\frac{\lambda_{p}\left\vert \Omega\right\vert }{\int_{\Omega}%
u_{q}^{p}dx}\left(  p-q\right)  .
\end{align*}
Taking into account (\ref{12}), there exists $R>0$ such that $\int_{\Omega
}u_{q}^{p}dx\geqslant R$ for all $q$ near to $p^{-}$. Thus,
\[
0\leqslant\mu\dfrac{\left\Vert v_{\mu,q}\right\Vert _{q}^{q}}{\left\Vert
v_{\mu,q}\right\Vert _{p}^{p}}-\mu_{q}\leqslant\frac{\lambda_{p}\left\vert
\Omega\right\vert }{R}\left(  p-q\right)  =K\left(  p-q\right)  .
\]

\end{proof}

\section{Some numerical results\label{Some numerical results}}

\subsection{Unit Balls}

In this section we present some numerical results in the unit ball of
dimensions $N=2,3,4$ applying Algorithm 1, since in this case $k_{p}$ is
explicitly known. Computations were performed on a Windows 7/ i5 - 4.0 GHz
platform, using the GCC compiler. The numerical approximations for the first
eigenpair were obtained choosing $\mu=k_{p}$ and taking $q=p-0.01$. Thus,
according to (\ref{kpball})%
\[
\mu=k_{p}=N\left(  \frac{p}{p-1}\right)  ^{p-1}.
\]

We recall from (\ref{greenball}) that for the unit ball the functions in the
sequence of iterates are radially ($r=\left\vert x\right\vert $) given by%

\[
v_{n+1}\left(  r\right)  =\int_{r}^{1}\left(  \int_{0}^{\theta}\left(
\frac{s}{\theta}\right)  ^{N-1}k_{p}v_{n}(s)^{q-1}ds\right)  ^{\frac{1}{p-1}%
}d\theta,\text{ \ with \ }v_{0}(r)\equiv\left(  1-\left\vert r\right\vert
^{\frac{p}{p-1}}\right)  .
\]
Thus, starting with the function%
\[
v_{0}(r)\equiv\left(  1-\left\vert r\right\vert ^{\frac{p}{p-1}}\right)
\]
we have implemented the sequence of iterates%
\begin{equation}
v_{n+1}\left(  r\right)  =\frac{pN^{\frac{1}{p-1}}}{p-1}\int_{r}^{1}\left(
\int_{0}^{\theta}\left(  \frac{s}{\theta}\right)  ^{N-1}v_{n}(s)^{(p-1.01)}%
ds\right)  ^{\frac{1}{p-1}}d\theta\label{vnball}%
\end{equation}
which, after normalized by the sup norm, should be close to the normalized
first eigenfunction $e_{p}.$

The first eigenvalue was approximated by the sequence
\[
\dfrac{N}{\left\Vert v_{n}\right\Vert _{\infty}^{0.01}}\left(  \frac{p}%
{p-1}\right)  ^{p-1}%
\]
given by (\ref{lamdaq}).

In order to compute sequence (\ref{vnball}) we mixed the composite Simpson and
trapezoidal methods on a $101$ points mesh for computation of the associated
integrals. We adopted
\begin{equation}
\frac{\Vert v_{n+1}-v_{n}\Vert_{\infty}}{\Vert v_{n}\Vert_{\infty}}<10^{-9}%
\end{equation}
as a stopping criterion.

At Table~\ref{tab:balls}, the results for the first eigenvalue of the
$p$-Laplacian for values of $p$ ranging from $1.1$ to $4.0$ for the unit balls
of dimensions $N=2,$ $3$ and $4$ are displayed and truncated at the fourth
decimal place. The results compare very well with the ones presented in
\cite{BEM1} up to the second decimal digit.

\begin{table}[ptb]
\caption{First eigenvalues of the $p$-Laplacian on the unit ball.}%
\label{tab:balls}
\centering
\begin{tabular}
[c]{lllllllll}\hline
$p$ & $N=2$ & $N=3$ & $N=4$ &  & $p$ & $N=2$ & $N=3$ & $N=4$\\\cline{1-4}%
\cline{6-9}%
1.1 & 2.5666 & 3.8665 & 5.17607 &  & 2.6 & 8.08856 & 14.9747 & 23.8345\\
1.2 & 2.9601 & 4.5026 & 6.0797 &  & 2.7 & 8.50354 & 15.9521 & 25.672\\
1.3 & 3.3182 & 5.1098 & 6.97306 &  & 2.8 & 8.92654 & 16.9646 & 27.6004\\
1.4 & 3.6637 & 5.71889 & 7.89478 &  & 2.9 & 9.35759 & 18.013 & 29.6225\\
1.5 & 4.0053 & 6.3419 & 8.86046 &  & 3.0 & 9.79673 & 19.0977 & 31.7409\\
1.6 & 4.3477 & 6.98495 & 9.87865 &  & 3.1 & 10.244 & 20.2194 & 33.9581\\
1.7 & 4.6932 & 7.65165 & 10.955 &  & 3.2 & 10.6994 & 21.3785 & 36.2769\\
1.8 & 5.0434 & 8.34438 & 12.094 &  & 3.3 & 11.163 & 22.5755 & 38.6999\\
1.9 & 5.3993 & 9.06487 & 13.2991 &  & 3.4 & 11.6347 & 23.8111 & 41.2298\\
2.0 & 5.7616 & 9.81443 & 14.5735 &  & 3.5 & 12.1146 & 25.0856 & 43.8694\\
2.1 & 6.1308 & 10.5942 & 15.9202 &  & 3.6 & 12.6027 & 26.3997 & 46.6213\\
2.2 & 6.5071 & 11.405 & 17.3421 &  & 3.7 & 13.099 & 27.7539 & 49.4884\\
2.3 & 6.8909 & 12.2478 & 18.8418 &  & 3.8 & 13.6034 & 29.1486 & 52.4734\\
2.4 & 7.2823 & 13.1232 & 20.422 &  & 3.9 & 14.1161 & 30.5844 & 55.5792\\
2.5 & 7.6815 & 14.0319 & 22.0855 &  & 4.0 & 14.6369 & 32.0618 &
58.8085\\\hline
\end{tabular}
\end{table}

Graphs of some eigenfunctions generated by the inverse iteration of sublinear
supersolutions are presented in Figures \ref{fig2GraficoAutoFuncoes},
\ref{fig3GraficoAutoFuncoes} and \ref{fig4GraficoAutoFuncoes} for $N=2,3$ and
$4,$ respectively. In these graphs it is possible to observe the asymptotic
behavior of the $L^{\infty}$-normalized eigenfunctions $e_{p}$ with respect to
$p$ for both cases: $p\rightarrow1^{-}$ and $p\rightarrow\infty.$ The
eigenfunctions $e_{p}$ converge to the characteristic function of the ball as
$p\rightarrow1^{+}$ (see \cite{KF}). On the other hand (see \cite{JLM}), as
$p\rightarrow\infty$ these functions converge to the distance function to the
boundary, which in this case is $1-\left\vert x\right\vert .$

\begin{figure}[ptb]
\begin{center}
\includegraphics[height = 5cm]{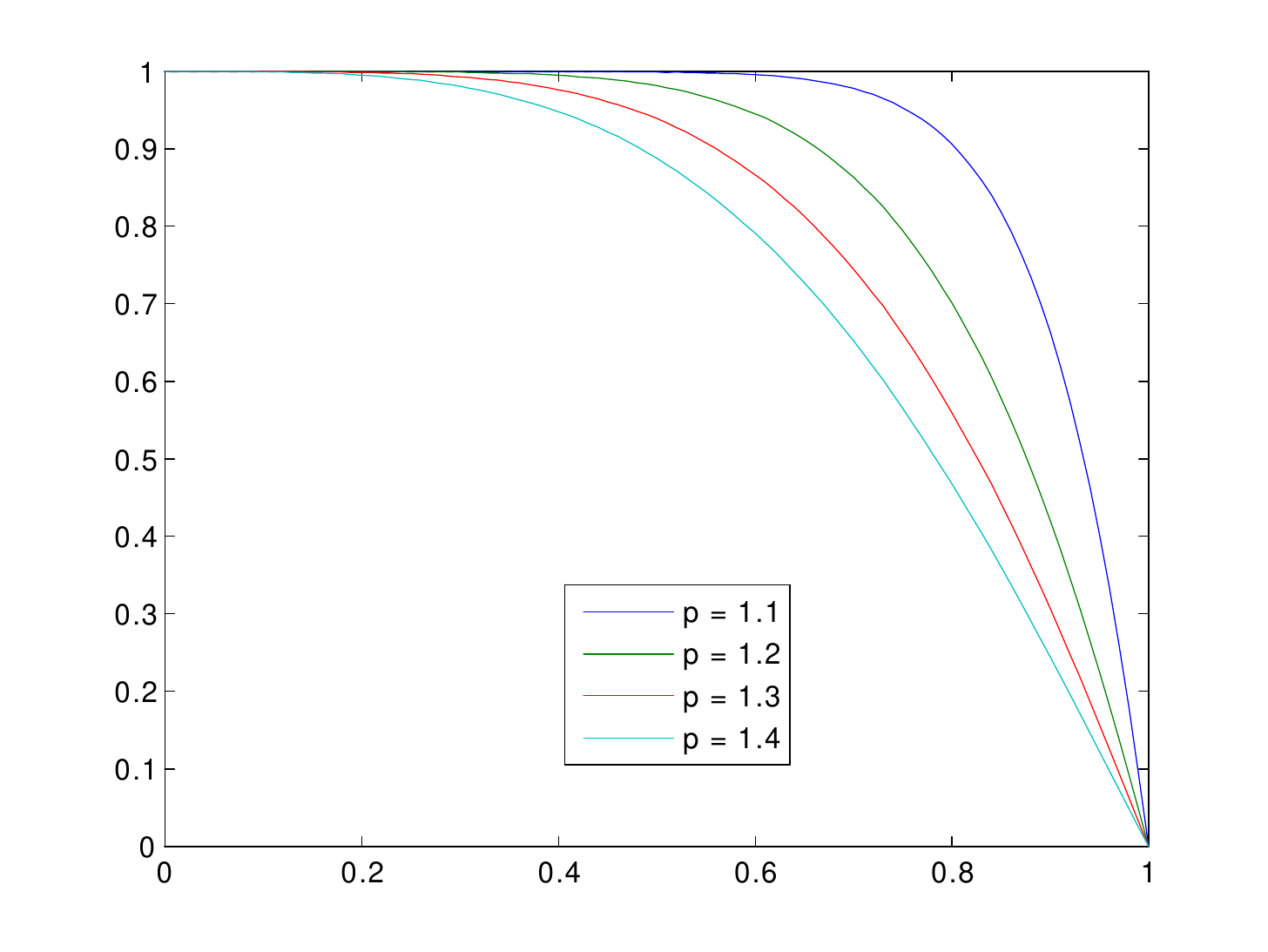}
\includegraphics[height = 5cm]{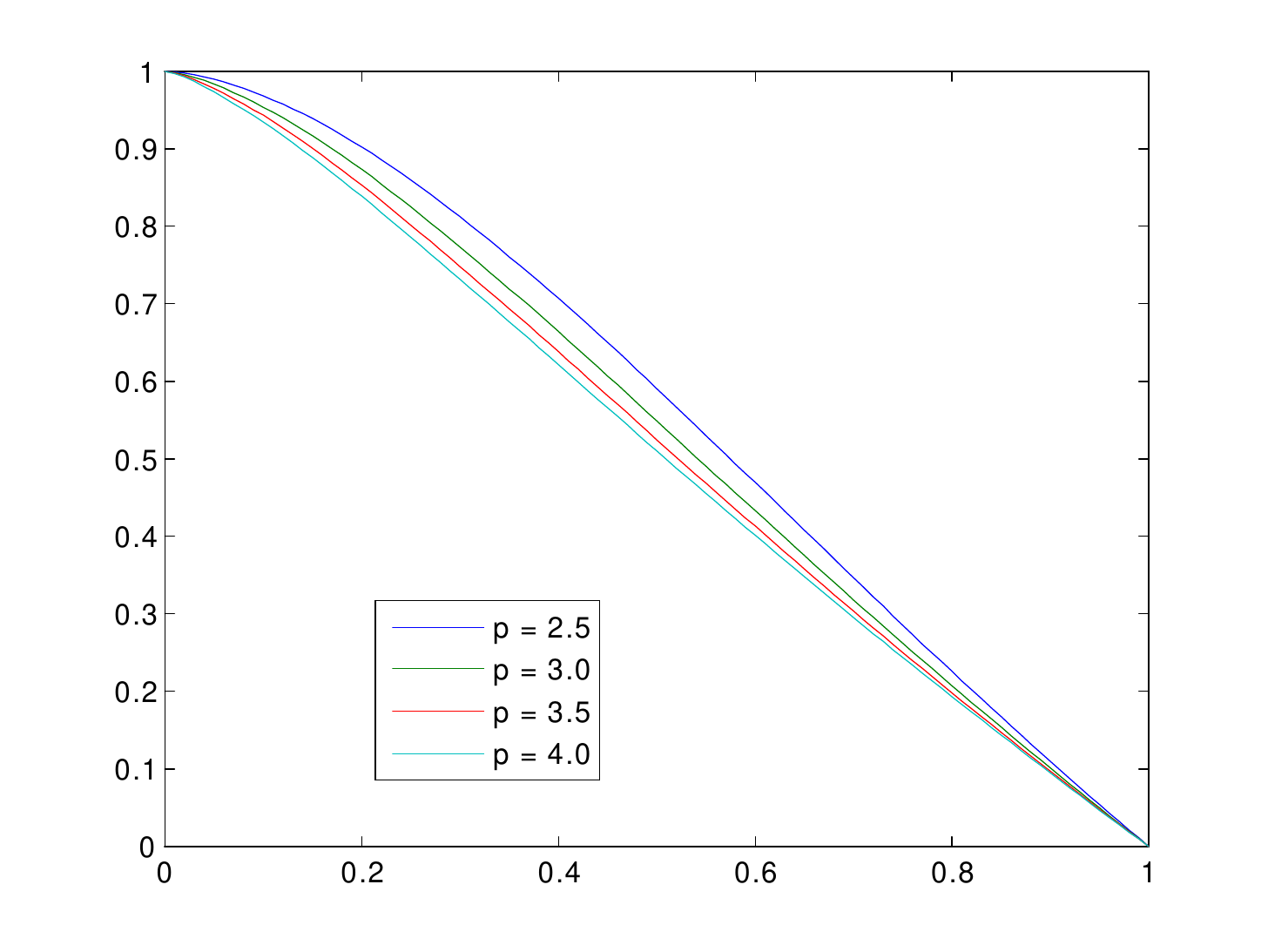}
\end{center}
\caption{Radial profiles of the first eigenfunction for unit ball when $N=2$,
$p=1.1,1.2,1.3,1.4$ (left) and $p=2.5,3.0,3.5,4.0$ (right).}%
\label{fig2GraficoAutoFuncoes}%
\end{figure}

\begin{figure}[ptb]
\begin{center}
\includegraphics[height = 5cm]{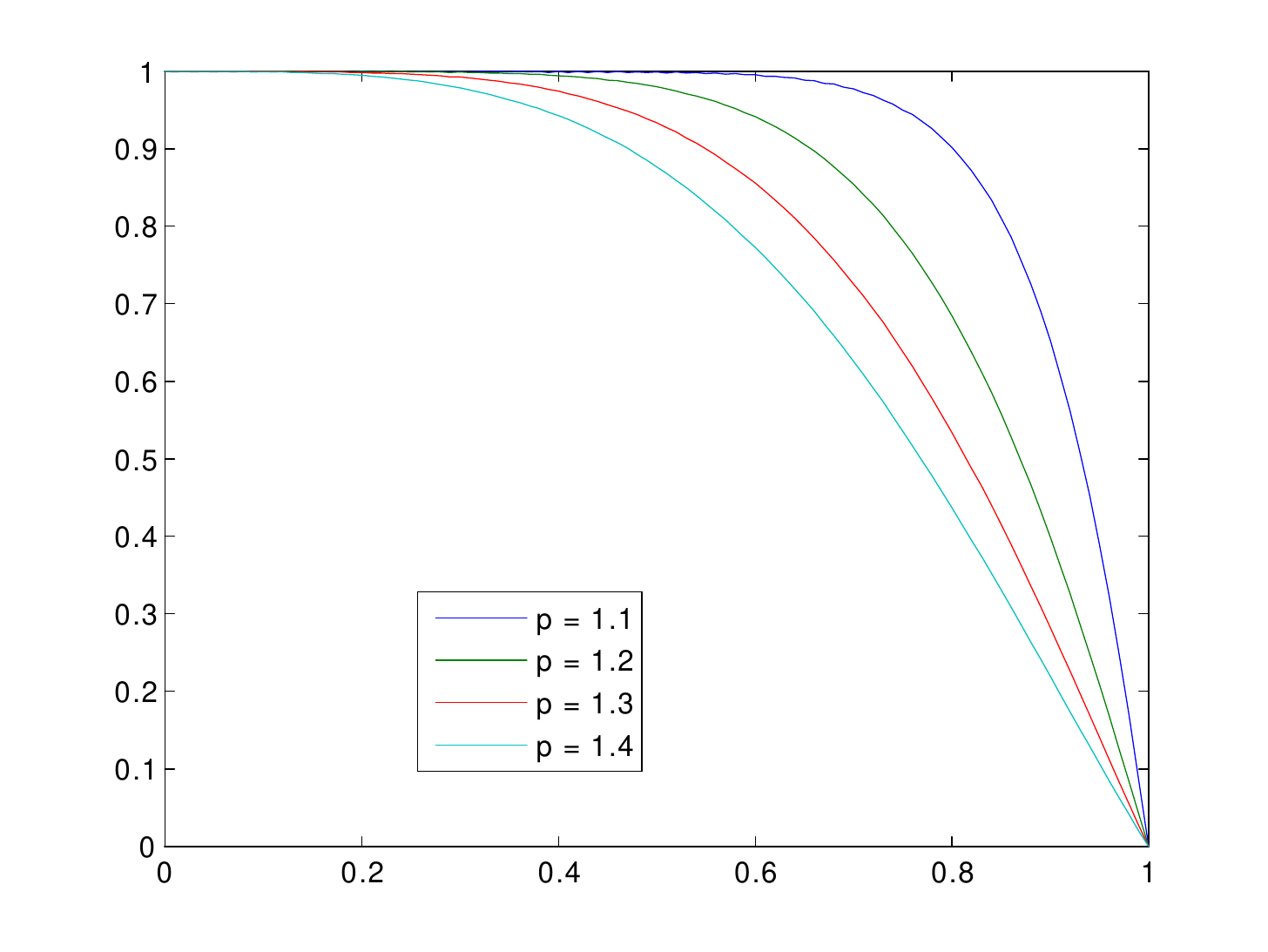}
\includegraphics[height = 5cm]{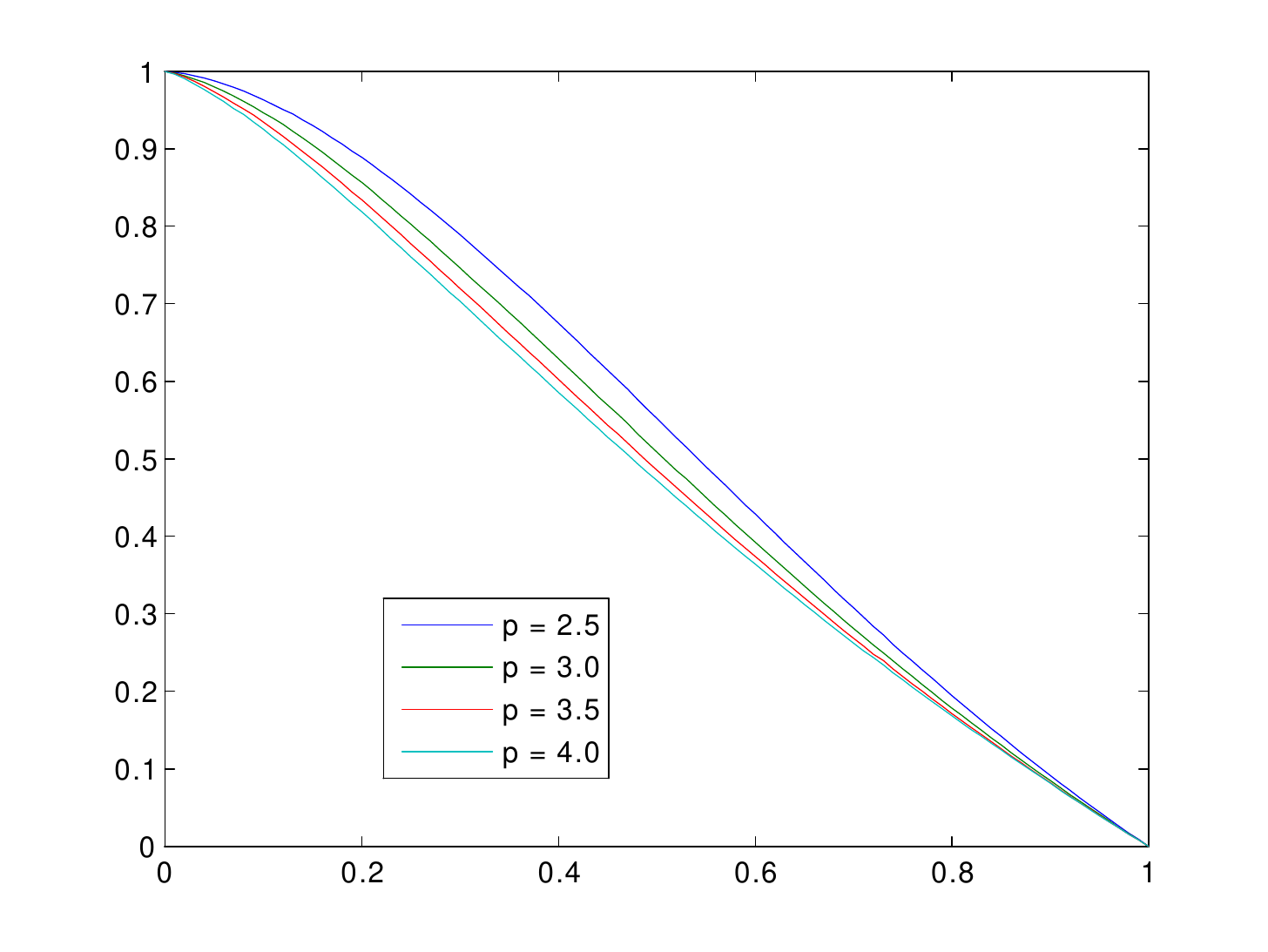}
\end{center}
\caption{Radial profiles of the first eigenfunction for unit ball when $N=3$,
$p=1.1,1.2,1.3,1.4$ (left) and $p=2.5,3.0,3.5,4.0$ (right).}%
\label{fig3GraficoAutoFuncoes}%
\end{figure}

\begin{figure}[ptb]
\begin{center}
\includegraphics[height = 5cm]{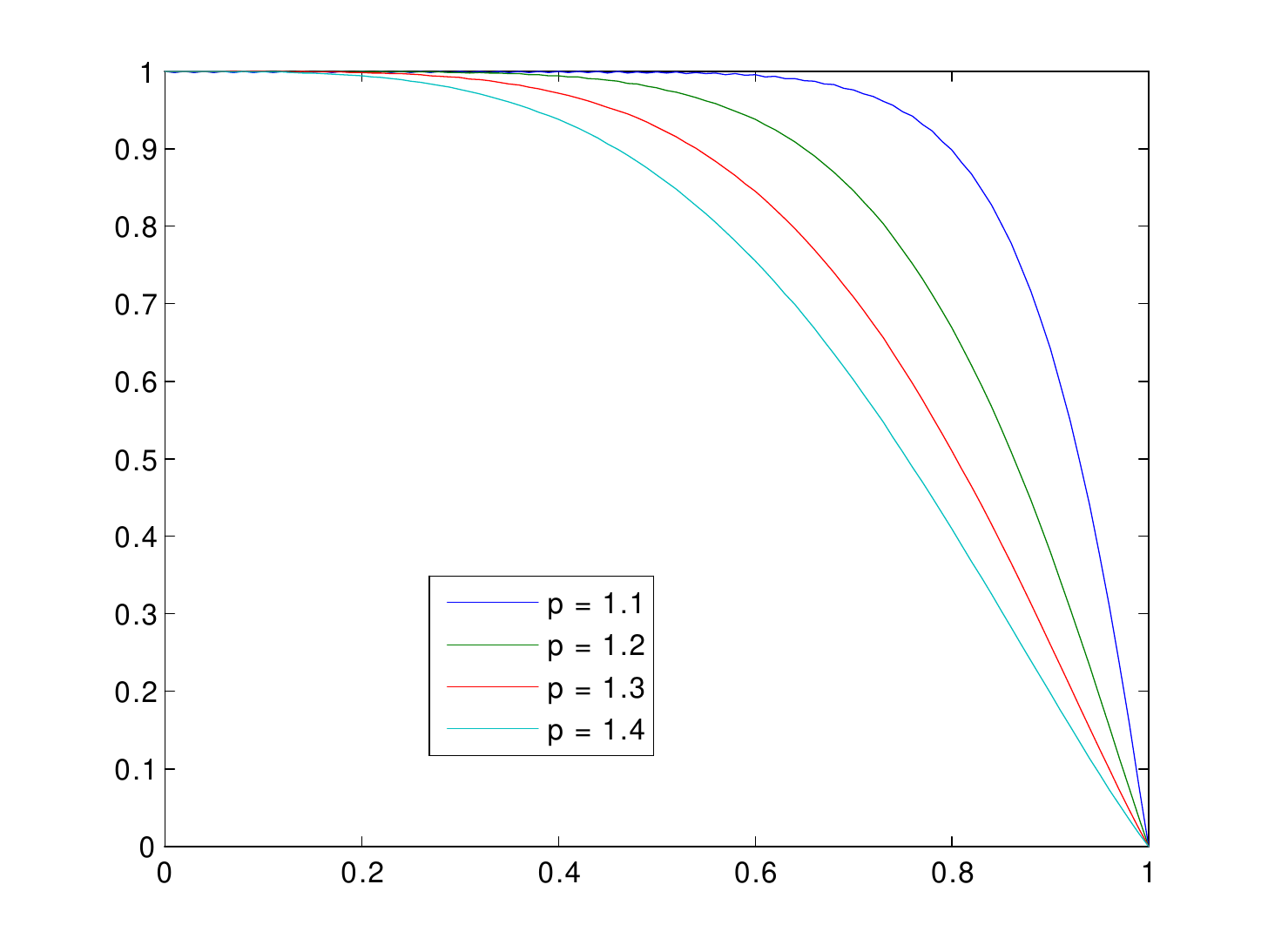}
\includegraphics[height = 5cm]{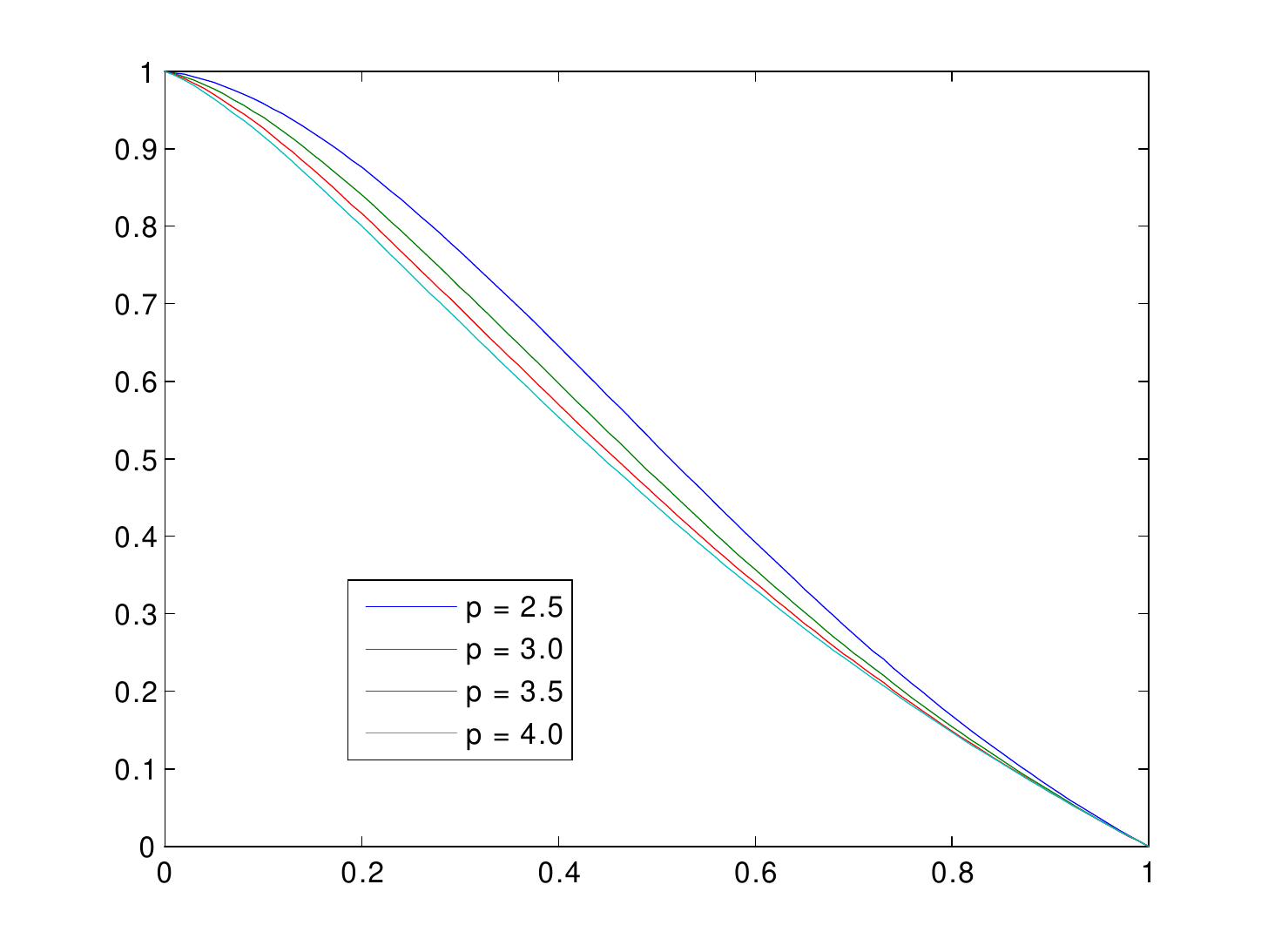}
\end{center}
\caption{Radial profiles of the first eigenfunction for unit ball when $N=4$,
$p=1.1,1.2,1.3,1.4$ (left) and $p=2.5,3.0,3.5,4.0$ (right).}%
\label{fig4GraficoAutoFuncoes}%
\end{figure}

Figure \ref{fig5GraficoConcavity} illustrates the log concavity of the
eigenfunctions $e_{p}.$ Note from Figures \ref{fig2GraficoAutoFuncoes},
\ref{fig3GraficoAutoFuncoes} and \ref{fig4GraficoAutoFuncoes} that each
eigenfunction $e_{p}$ seems to be convex near the boundary ($r=1$)$.$ However,
$\log(e_{p})$ is surely concave for convex domains, as proved in \cite{Saka}.

\begin{figure}[ptb]
\begin{center}
\includegraphics[height = 6.5cm,width = 10cm]{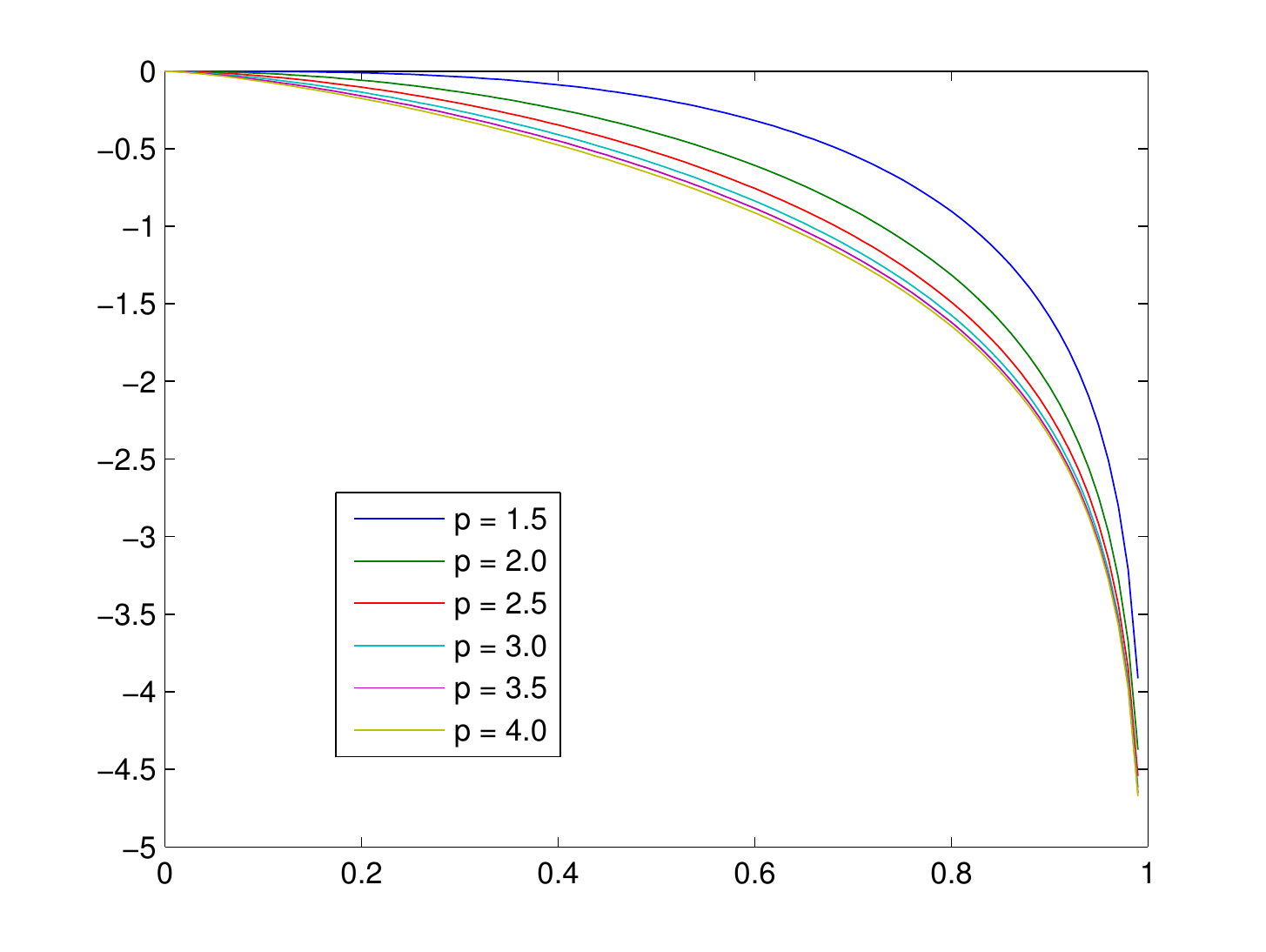}
\includegraphics[height = 6.5cm,width = 10cm]{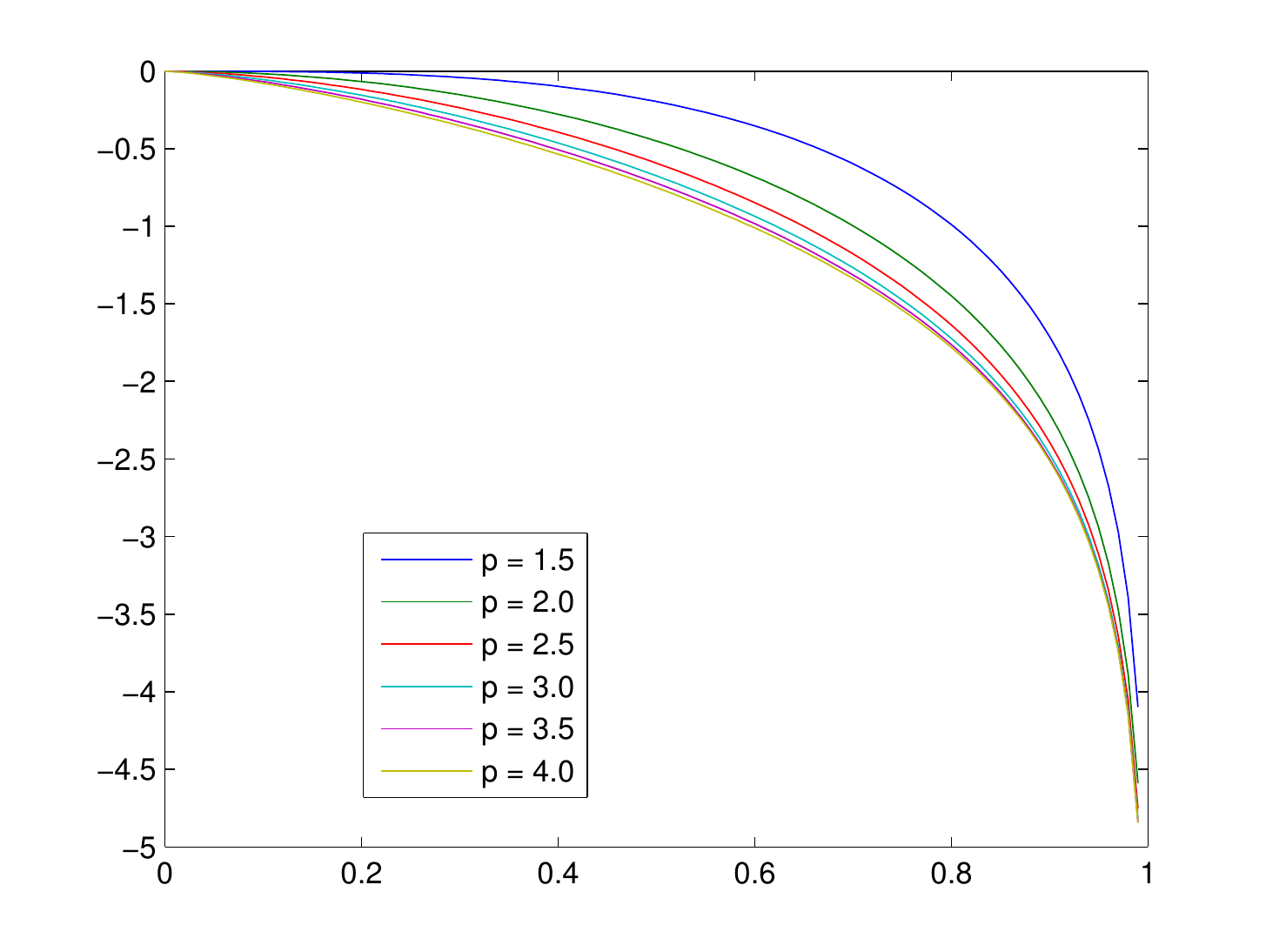}
\includegraphics[height = 6.5cm,width = 10cm]{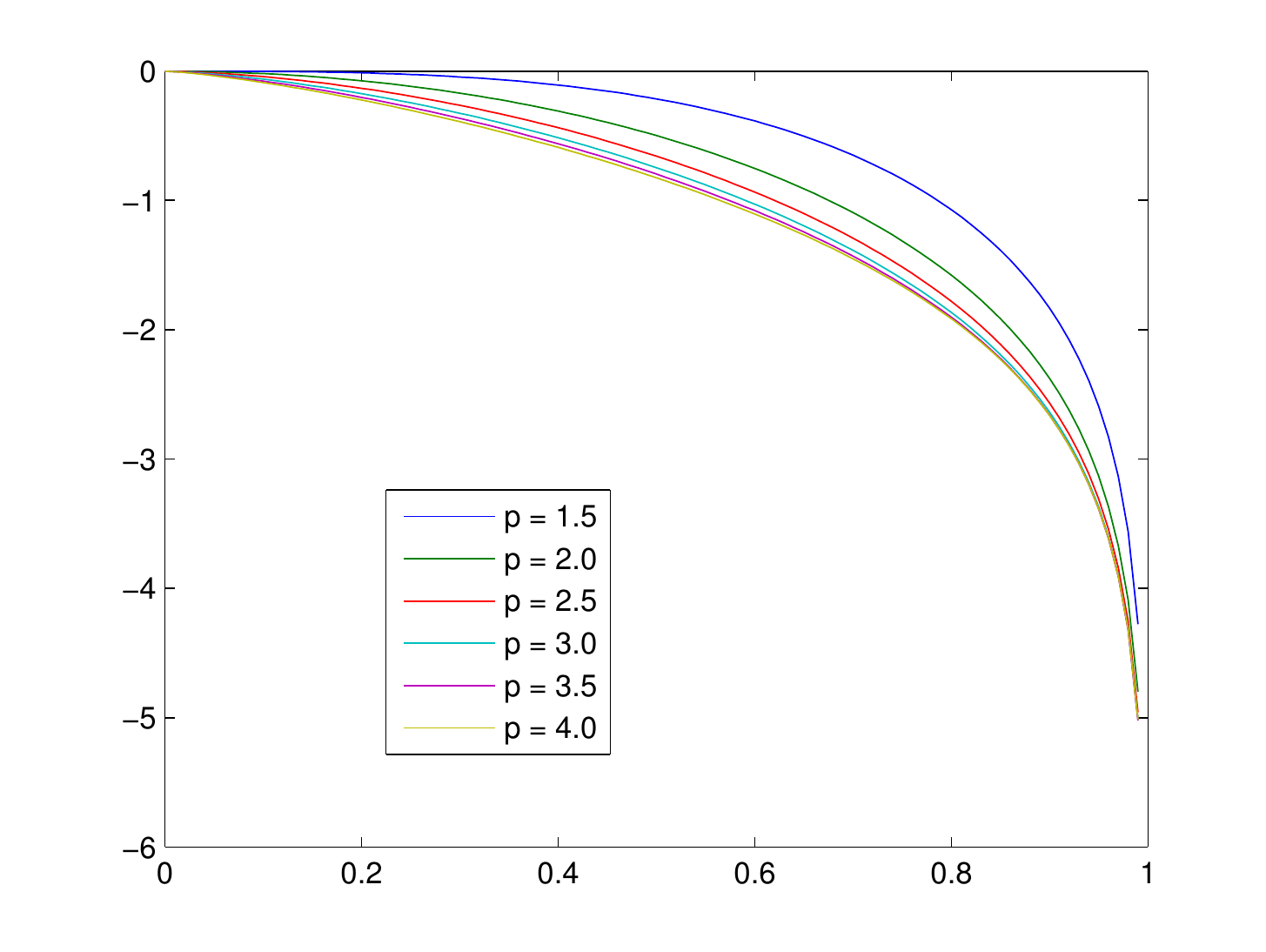}
\end{center}
\caption{Graphs of $\log(e_{p})$ versus $p$ for the $N$-dimensional unit ball,
and $N=2$ (above), $N=3$ (center), $N=4$ (below), and
$p=1.5,2.0,2.5,3.0,3.5,4.0.$ }%
\label{fig5GraficoConcavity}%
\end{figure}

\begin{figure}[ptb]
\begin{center}
\includegraphics[height = 6cm,width = 12cm]{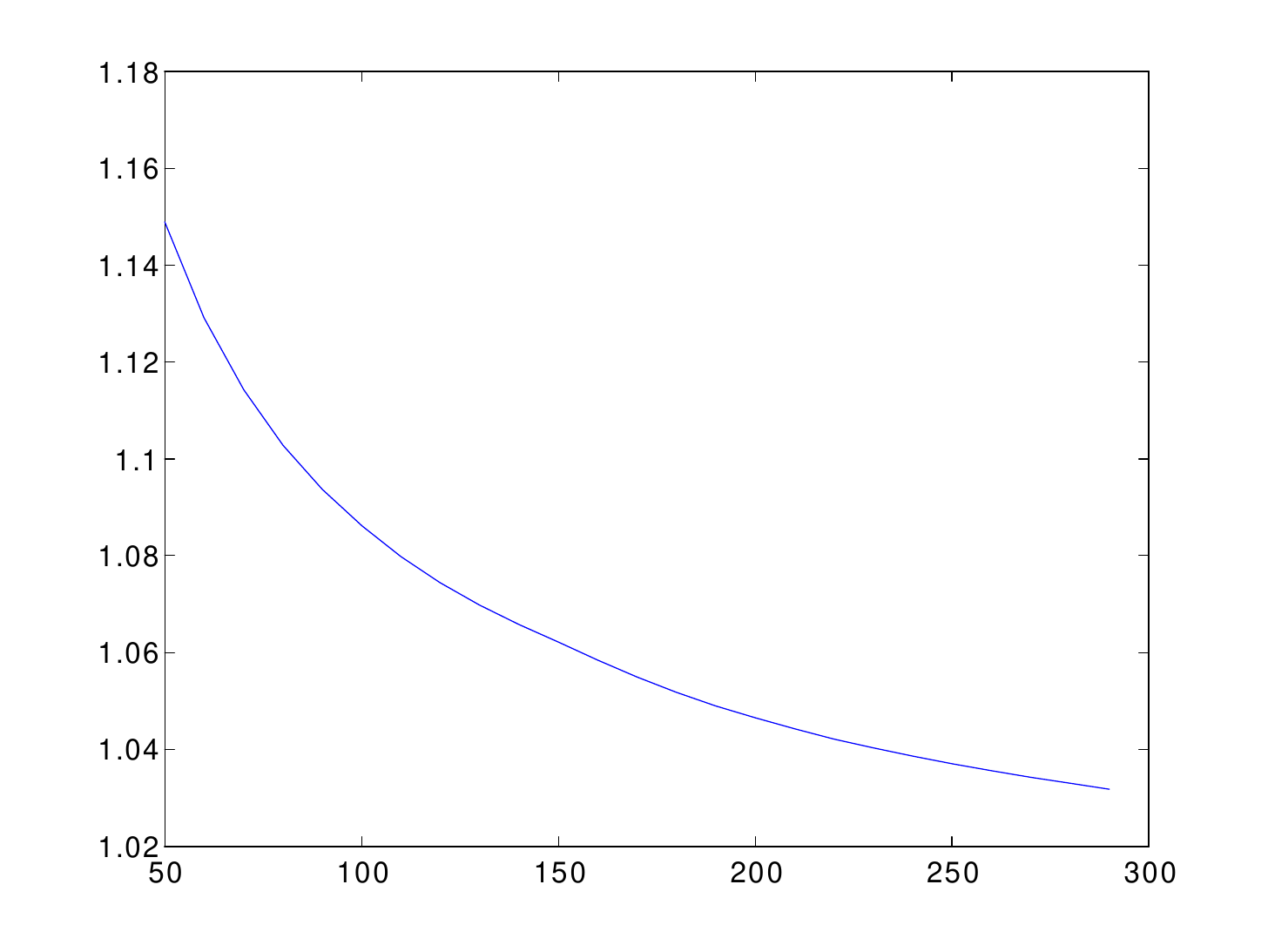}
\includegraphics[height = 6cm,width = 12cm]{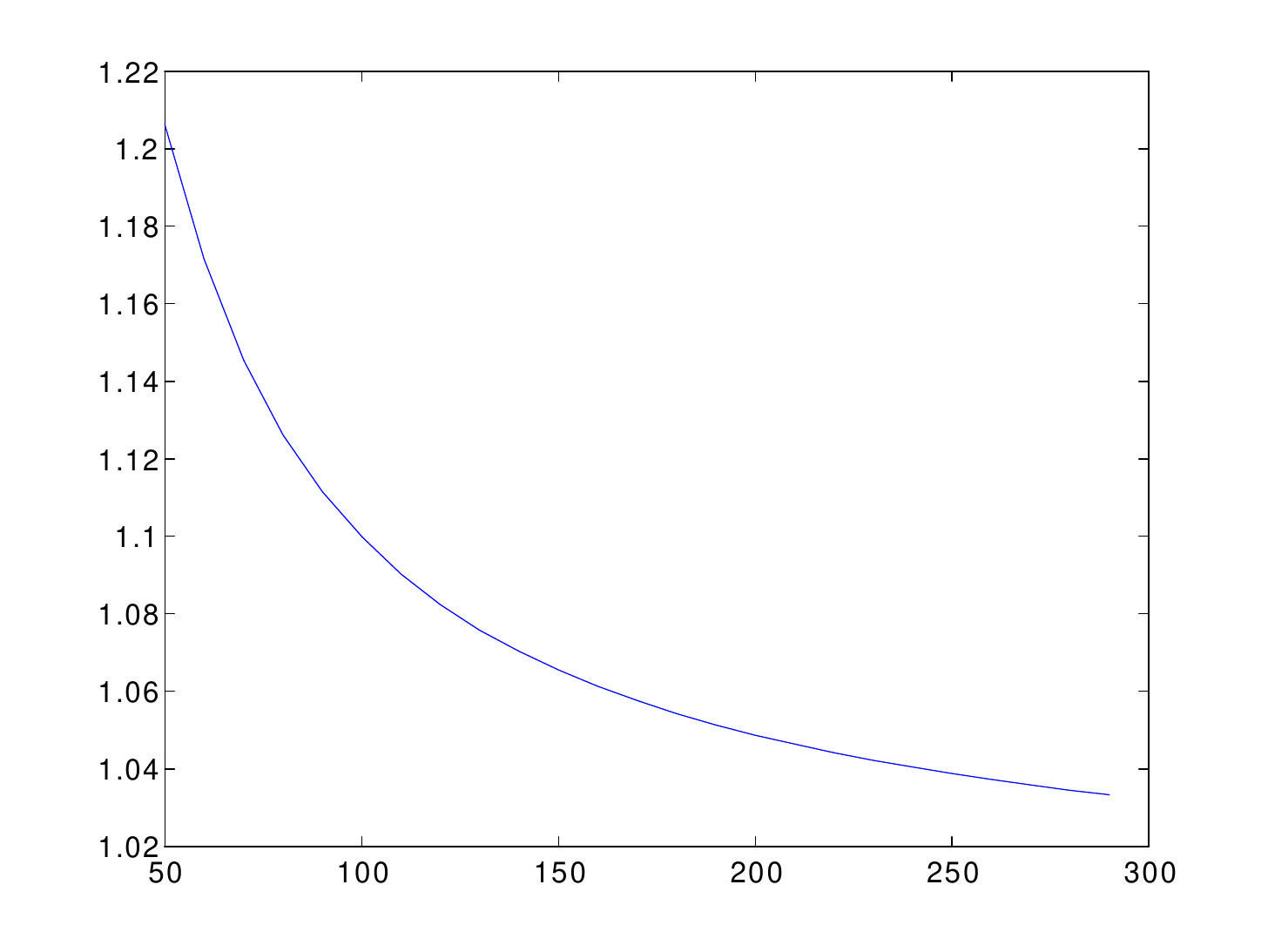}
\includegraphics[height = 6cm,width = 12cm]{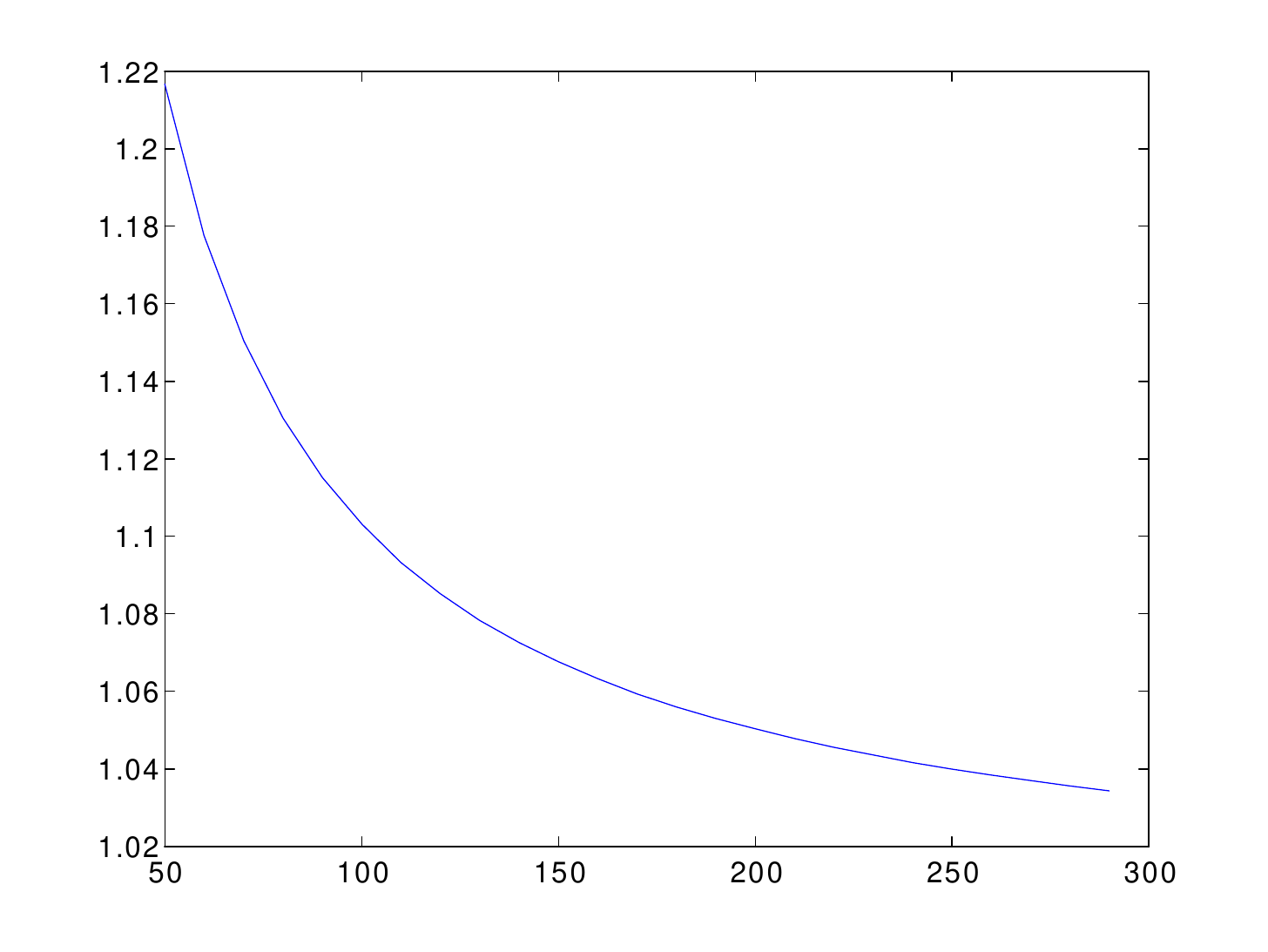}
\end{center}
\caption{Graphs of $\sqrt[p]{\lambda_{p}}$ versus $p$ for the $N$-dimensional
unit ball, and $N=2$ (above), $N=3$ (center), $N=4$ (below), from $p=50$ to
$p=290$, step 10.}%
\label{fig6GraficoRoot}%
\end{figure}

In Figure \ref{fig6GraficoRoot} we see that $\sqrt[p]{\lambda_{p}}$ approaches
$1$ as $p$ increases, which is coherent with the following known asymptotic
behavior (see \cite{JLM}): $\lim\limits_{p\rightarrow\infty}\sqrt[p]%
{\lambda_{p}}=1/R$ where $R$ is the inradius of the domain (that is, the
radius of the largest ball that lies within the domain). Moreover, one
observes from Table~\ref{tab:balls} that $\lambda_{p}$ approaches the value
$N$ of the dimension as $p\rightarrow1^{-}.$ It is known (see \cite{KF}) that
$\lambda_{p}$ tends to $N/R,$ if the domain is a ball of radius $R.$

Finally, for comparison we show in Figure \ref{fig1GraficoAutovalores} graphs
of $p$ versus $\lambda_{p}$ obtained in two ways: one of them through the
method proposed in \cite{BEM1} which is directly based on the inverse power
method (IPM), while the other is the inverse iteration of sublinear
supersolutions (IISS) as developed in the present work.

\begin{figure}[ptb]
\begin{center}
\includegraphics[height = 4cm]{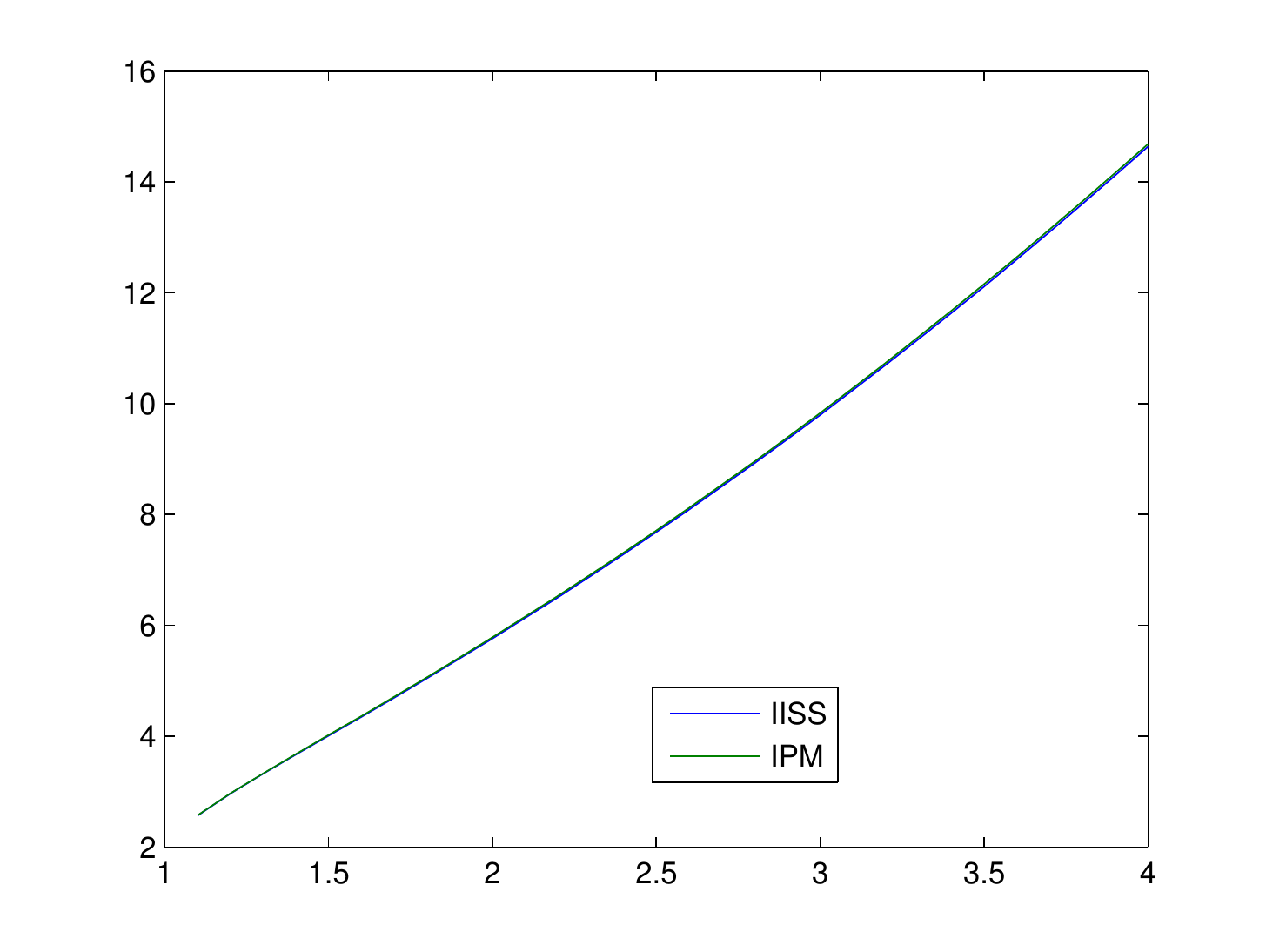}
\includegraphics[height = 4cm]{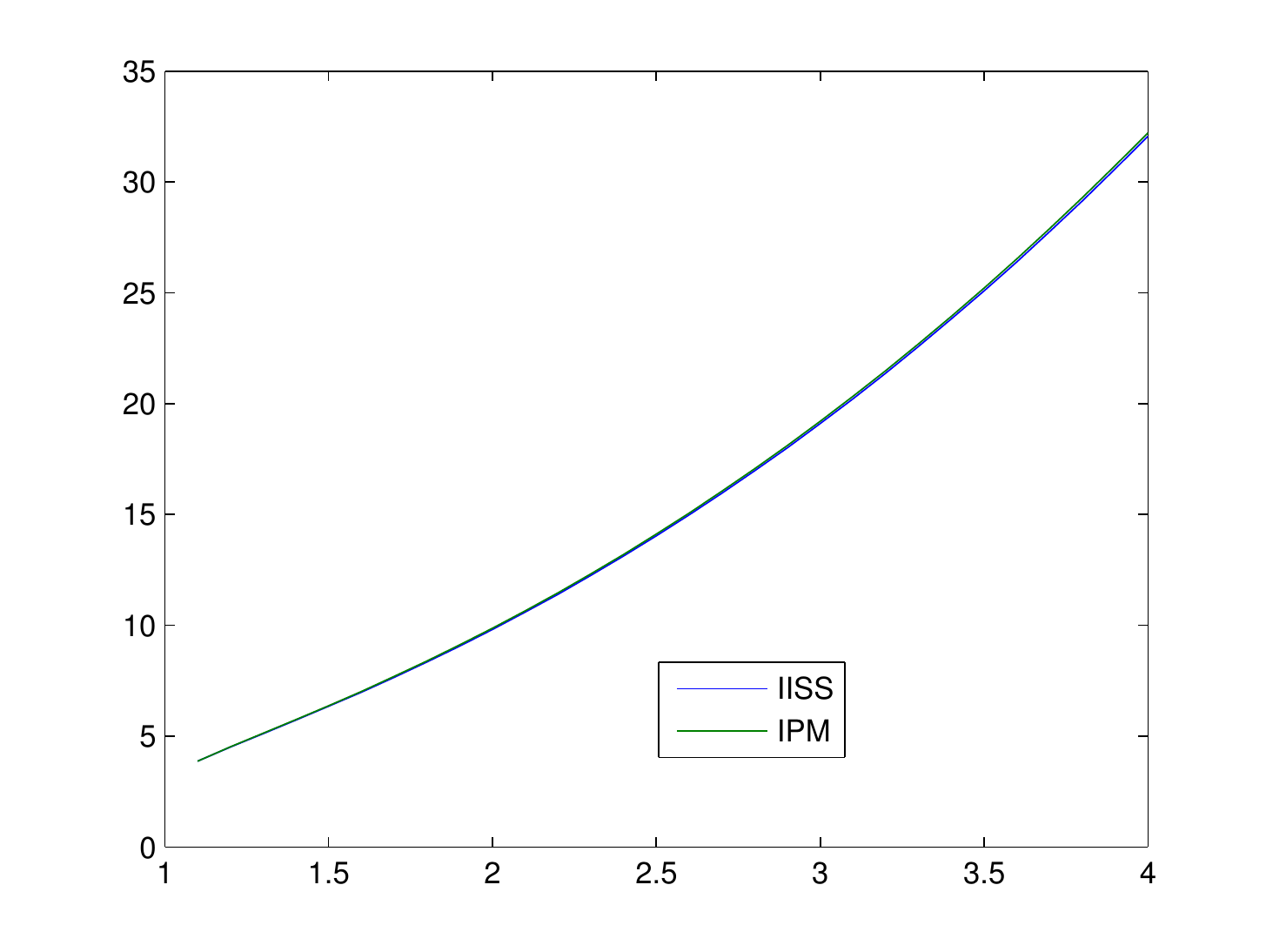}
\includegraphics[height = 4cm]{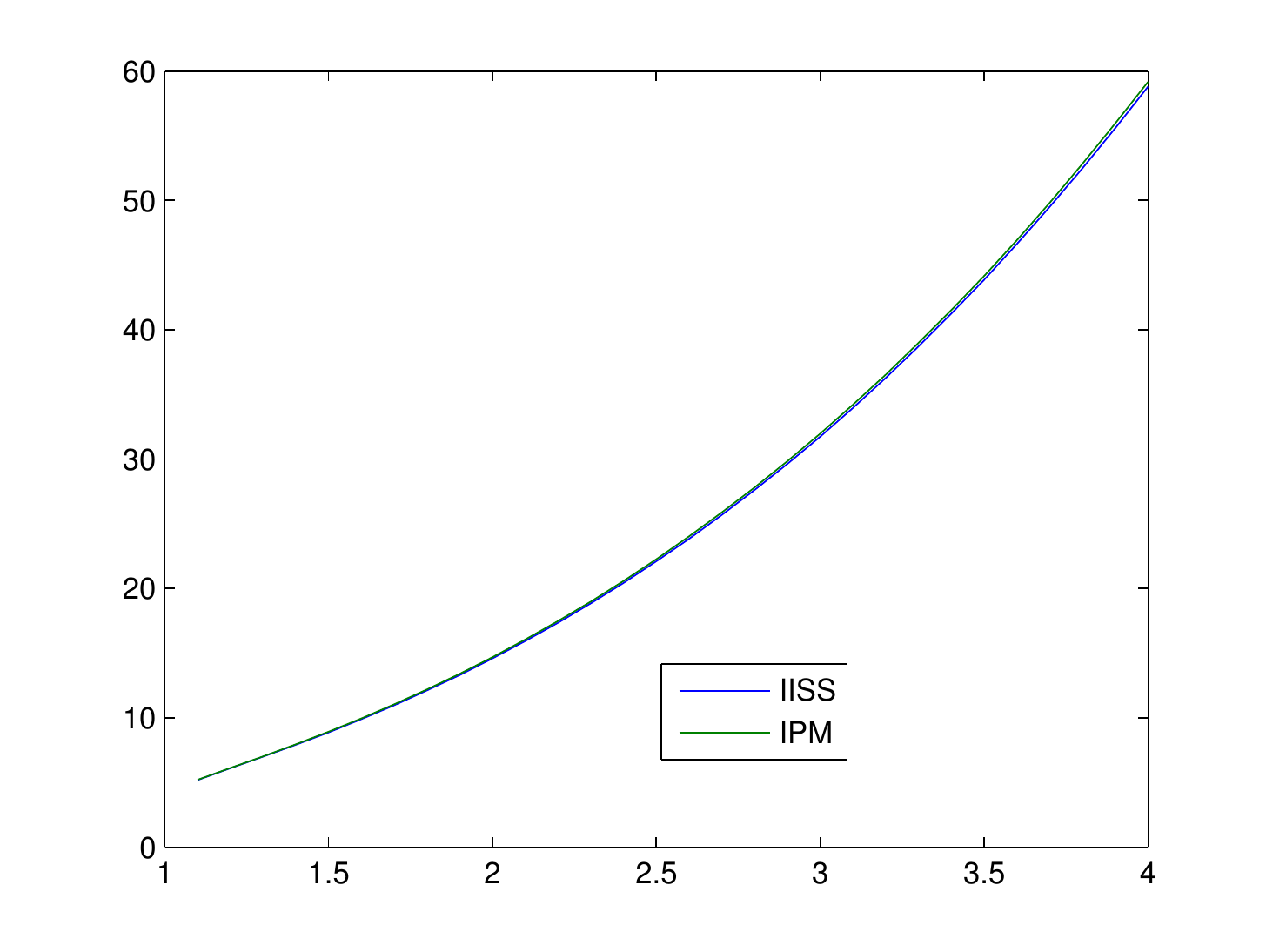}
\end{center}
\caption{Graphs of $\lambda_{p}$ versus $p$ for the N-dimensional unit ball
and $N=2$ (left), $N=3$ (center), $N=4$ (right).}%
\label{fig1GraficoAutovalores}%
\end{figure}

\vspace{1cm}

\subsection{Square, Cube and Torus}

To compute eigenvalues on more general domains, we use a $p$-version finite
element discretization on unstructured hexahedral meshes. The discrete
equations are solved with PETSc \cite{petsc-user-ref} using a Newton-Krylov
method in which a matrix associated with a lowest-order discretization is
assembled for preconditioning, while the high-order operator is applied in
unassembled form (see \cite{brown2010dohp} for details). For these more
complicated domains we apply Algorithm 2. This produces the system
\[
-\nabla\cdot\Big( \big(\epsilon^{2} + \left|  \nabla\phi_{m+1} \right|  ^{2}
\big)^{\frac{p-2}{2}} \nabla\phi_{m+1} \Big) = \left(  \frac{\phi_{m}%
}{\left\|  \phi_{m} \right\|  _{\infty}} \right)  ^{q-1}%
\]
where $\epsilon=10^{-5}$ is the
regularization used to avoid the singularity or degeneracy at $\nabla\phi= 0$.
The initial guess for the Newton iteration is taken to be $\phi_{m}$ which
leads to very fast convergence in the terminal phase. To solve 2D problems
with the 3D discretization, homogeneous Neumann boundary conditions are
imposed on both faces in the $z$ direction. The source code is publicly
available from \url{https://github.com/jedbrown/dohp}.

Table~\ref{tab:cube} shows computed eigenvalues for the unit square and unit
cube. These solutions were computed using $Q_{5}$ elements and are as accurate
as double precision rounding error for the smooth solutions in the $p=2$ case.
The accuracy of the discretization for a given smooth solution has been
verified to be essentially independent of $p$ using the method of manufactured
solutions. This indicates that the primary source of error in
Table~\ref{tab:cube} is interpolation error, as usual for finite element methods.

\begin{table}[ptb]
\caption{Eigenvalues of the $p$-Laplacian on the unit square and cube. The 2D
results use a $10\times10$ mesh of $Q_{5}$ elements, the 3D results use a
$6\times6\times6$ mesh of $Q_{5}$ elements. For $p=2$, the exact solutions
$2\pi^{2}$ and $3\pi^{2}$ are available so we show the error in the Reference
column; these cases are denoted by [*]. }%
\label{tab:cube}
\centering
\begin{tabular}
[c]{llllllll}\hline
&  & $2D$ &  &  &  & $3D$ & \\\cline{2-4}\cline{6-8}%
$p$ & Computed &  & Reference &  & Computed &  & Reference\\\hline
1.2 & 6.195550328210643 &  &  &  & 8.642315135978254 &  & \\
1.5 & 10.07201415299496 &  & 10.0722~\cite{bognar2003} &  &
14.47791516619582 &  & \\
1.75 & 14.28146165697044 &  & 14.2815~\cite{bognar2003} &  &
20.96672431961172 &  & \\
2 & 19.73920880217817 &  & 5.36$\times$10$^{-13}$[*] &  & 29.60881320326431 &
& 3.77$\times$10$^{-12}$[*]\\
2.2 & 25.24862830212583 &  & 25.2412~\cite{BEM1} &  & 38.51651963302274 &  &
\\
2.5 & 35.94868349730170 &  & 35.9493~\cite{bognar2003} &  &
56.19031685699854 &  & \\
3 & 62.75762286200781 &  & 62:7633~\cite{bognar2003} &  & 101.8697977481977 &
& \\
4 & 176.5980821441738 &  & 176.693~\cite{bognar2003} &  & 306.1647710559179 &
& \\
5 & 463.8206306371868 &  &  &  & 849.9777670614186 &  & \\\hline
\end{tabular}
\end{table}

Figure~\ref{fig:torus} shows computed eigenfunctions for $p=1.2$ and $p=5$ on
a torus. The unstructured hexahedral mesh was created with CUBIT version
13.0~\cite{blacker1994cmg} using the commands
\begin{verbatim}
  create torus major radius 1 minor radius 0.4
  webcut volume all with plane xplane offset 0
  mesh volume 1 2
\end{verbatim}

\noindent and a $Q_{2}$ discretization was used. The computed eigenvalues are
$\lambda_{p} = 7.800846$ for $p=1.2$ and $\lambda_{p} = 2064.08$
for $p=5$.

\begin{figure}[ptb]
\begin{center}
\includegraphics[width=0.8\textwidth]{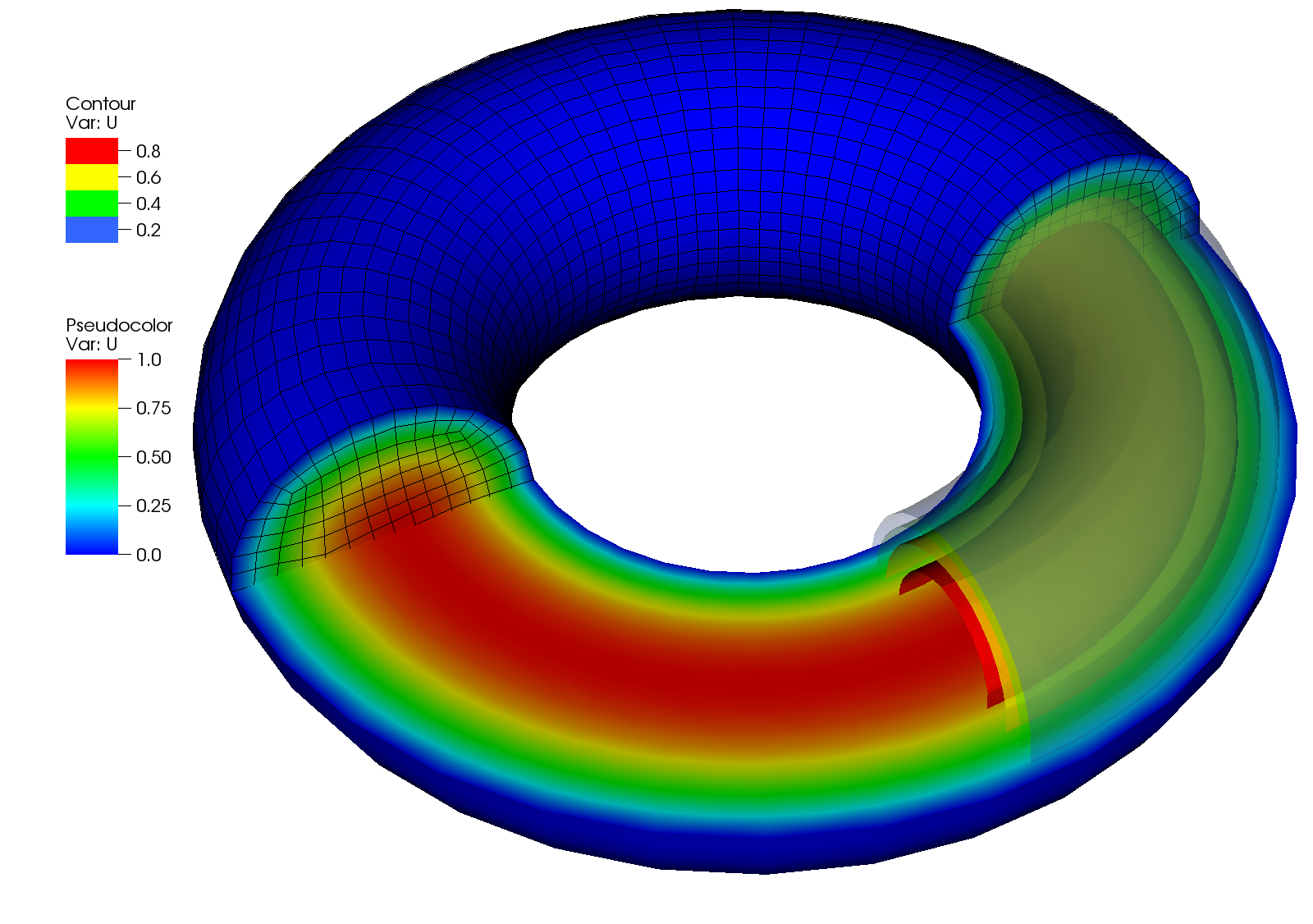} \newline%
\includegraphics[width=0.8\textwidth]{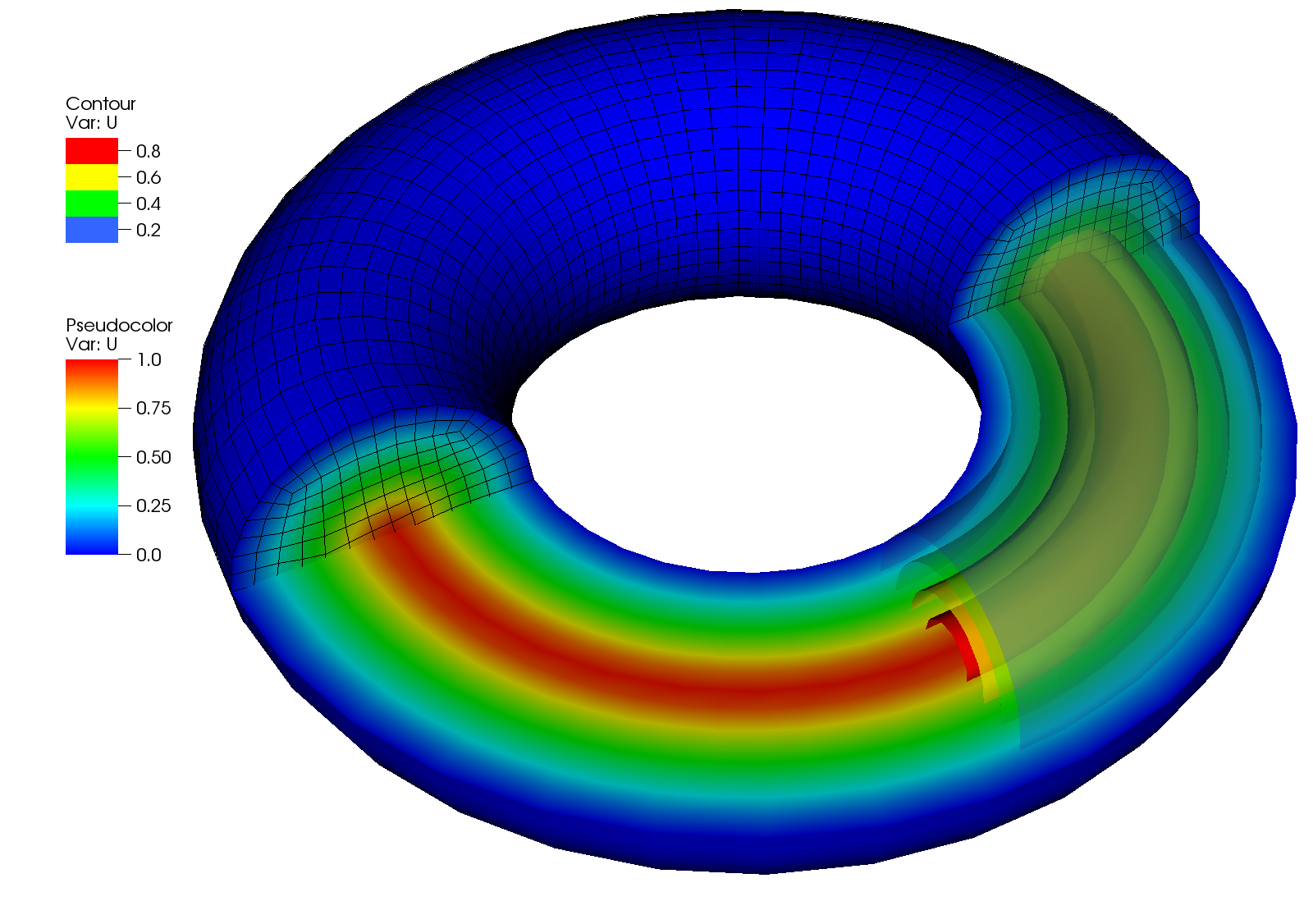}
\end{center}
\caption{Eigenfunctions of the $p$-Laplacian for $p=1.2$ (top) and $p=5$
(bottom) computed on a torus with major radius $1$ and minor radius
$0.4$.}%
\label{fig:torus}%
\end{figure}

Experimental evidence suggests that Algorithm 2 converges with $q=p$, but we
have only been able to prove convergence for $q < p$. It is unknown whether
the iteration will break down for some domain when $q=p$, but one can always
compute with $q<p$ in which case Theorem~8 guarantees convergence with an
error less than $K(p-q)$ for some positive constant $K$ depending only on the
domain and $p$. Table~\ref{tab:qtop} shows numerical evidence of this result
and quantifies $K$ for the unit cube with $p=1.5$ and $p=3$.

\begin{table}[ptb]
\caption{Convergence of the computed eigenvalue $\lambda_{p,q} \to\lambda
_{p}^{-}$ as $q\to p^{-}$ for the unit cube using a $4\times4\times4$ mesh
with $Q_{5}$ discretization.}%
\label{tab:qtop}
\centering
\begin{tabular}
[c]{llllllll}\hline
&  & $p=1.5$ &  &  &  & $p=3$ & \\\cline{2-4}\cline{6-8}%
$p-q$ & $\lambda_{p,q}$ &  & $\lambda_{p}-\lambda_{p,q}$ &  & $\lambda_{p,q}$
&  & $\lambda_{p}-\lambda_{p,q}$\\\hline
10$^{-1}$ & 13.797661713072 &  & 6.7943$\times$10$^{-1}$ &  &
96.414190427672 &  & 5.4559\\
10$^{-2}$ & 14.405694866696 &  & 7.1393$\times$10$^{-2}$ &  &
101.31034449471 &  & 5.5975$\times$10$^{-1}$\\
10$^{-3}$ & 14.469912150762 &  & 7.1756$\times$10$^{-3}$ &  &
101.81397339451 &  & 5.6119$\times$10$^{-2}$\\
10$^{-4}$ & 14.476369813850 &  & 7.1792$\times$10$^{-4}$ &  &
101.86447907670 &  & 5.6133$\times$10$^{-3}$\\
10$^{-5}$ & 14.477015941408 &  & 7.1796$\times$10$^{-5}$ &  &
101.86953107554 &  & 5.6135$\times$10$^{-4}$\\
10$^{-6}$ & 14.477080557778 &  & 7.1796$\times$10$^{-6}$ &  &
101.87003628973 &  & 5.6135$\times$10$^{-5}$\\
0 & 14.477087737416 &  & - &  & 101.87009242480 &  & -\\\hline
\end{tabular}\end{table}

\begin{table}[ptb]
\caption{Convergence rate for inverse iteration applied to the torus with
$p=1.2$ and $p=5$. Each nonlinear solve is converged to a relative tolerance
of $10^{-8}$.}%
\label{tab:torusconv}
\centering
\begin{tabular}
[c]{clcl}%
\multicolumn{2}{c}{$p=1.2$} & \multicolumn{2}{c}{$p=5$}\\
\cmidrule(r){1-2} \cmidrule(l){3-4} {Newton its.} & {$\lambda_{p}$} & {Newton
its.} & {$\lambda_{p}$}\\
\midrule 37 & {torsion} & 18 & {torsion}\\
5 & 7.7670871 & 6 & 1628.81\\
4 & 7.7965212 & 4 & 1975.40\\
3 & 7.8003037 & 4 & 2043.11\\
3 & 7.8007802 & 3 & 2057.15\\
2 & 7.8008389 & 3 & 2061.17\\
2 & 7.8008456 & 3 & 2062.74\\
2 & 7.8008462 & 3 & 2063.35\\
&  & 3 & 2063.38\\
&  & 3 & 2063.98\\
&  & 3 & 2064.08\\
\bottomrule &  &  &
\end{tabular}
\end{table}


In practice, the total computational cost to solve the eigenvalue problem is
about twice that of only solving the torsion creep problem.
Table~\ref{tab:torusconv} shows the convergence of inverse iteration when the
Newton iteration at each step is started using the solution at the last
iteration. The initial guess for the torsion creep problem is zero, which
leads to a difficult nonlinear solve. The Newton iteration is guarded by a
cubic backtracking line search which is sufficient in this case; a parameter
continuation or grid sequencing is more robust. The torsion creep problem is
significantly easier to solve for less extreme values of $p$ or for larger
values of the regularization $\epsilon$. After the torsion creep problem has
been solved, a line search is no longer necessary and accurate estimates of
the eigenvalue can be obtained in a few more Newton iterations.


\section{Acknowledgments}

\noindent The authors would like to thank the support of FAPEMIG and CNPq.

\end{document}